\documentclass{amsart}
\usepackage{amssymb}
\usepackage[T1]{fontenc}
\usepackage{mathrsfs}
\usepackage{enumerate}
 \usepackage{hyperref}
 \usepackage{cleveref}
\usepackage{tikz}

\newtheorem{thm}{Theorem}
\newtheorem{defn}{Definition}
\newtheorem{fact}{Fact}
\newtheorem{lem}[thm]{Lemma}
\newtheorem{prop}[thm]{Proposition}
\newtheorem{cor}[thm]{Corollary}
\newtheorem{question}{Question}

\newtheorem{exam}{Example}

\newcommand{\pp}{\mathcal{P}}
\newcommand{\up}{\upharpoonright}
\newcommand{\ra}{\rightarrow}
\newcommand{\mc}{\mathcal}
\newcommand{\ms}{\mathscr}
\newcommand{\msa}{\mathscr{A}}
\newcommand{\msb}{\mathscr{B}}
\newcommand{\pah}{\mathscr{P}_{\mathscr{A}, H}}
\newcommand{\bfe}{\mathbf{E}}
\newcommand{\mct}{\mathcal{T}}
\newcommand{\Phit}{\Phi_\mathcal{T}}

\newcommand{\phipi}{\varphi(\pi, \mathcal{C}, \mathcal{R}, \mathcal{T}, \mathbf{E})}
\newcommand{\psipi}{\psi_0(\pi, \mathcal{C}, \mathcal{R}, \mathcal{T}, \mathbf{E})}

\newcommand{\ma}{{\rm MA}}
\newcommand{\map}{{\rm MA}(powerfully ccc)}

\makeatletter
\@addtoreset{case}{thm}
\@addtoreset{case}{lem}
\@addtoreset{case}{prop}
\makeatother

\begin{document}

 \title{Distinguishing Martin's axiom from its restrictions}
  \author{Yinhe Peng}
  \address{Academy of Mathematics and Systems Science, Chinese Academy of Sciences\\ East Zhong Guan Cun Road No. 55\\Beijing 100190\\China}
\email{pengyinhe@amss.ac.cn}

\thanks{The author was partially supported by a program of the Chinese Academy of Sciences.}


  \begin{abstract}
 We introduce an iteration of forcing notions satisfying the countable chain condition with minimal damage to a strong coloring. Applying this method, we prove that Martin's axiom  is strictly stronger than its restriction to forcing notions satisfying the countable chain condition in all finite powers. Our method shows also the finer distinction, that Martin's axiom is strictly stronger than its restriction to forcing notions whose squares satisfy the countable chain condition.
  \end{abstract}
  
  \begingroup
  \def\uppercasenonmath#1{}
    \maketitle
    \endgroup

    \section{Introduction}
     Solovay and Tennebaum \cite{ST}  introduced the method of iterated forcing to prove the consistency of Suslin's hypothesis. Then Martin formulated a forcing axiom Martin's axiom ({\rm MA}) whose consistency follows from Solovay-Tennebaum's method.
     
     For cardinal $\kappa$,  {\rm MA}$_{\kappa}$ is the following assertion:
  \begin{itemize}
  \item If $\mc{P}$ is a ccc partial order and $\{D_\alpha: \alpha<\kappa\}$ is a collection of dense subsets of $\mc{P}$, then there is a filter $G\subset \mc{P}$ such that $G\cap D_\alpha\neq \emptyset$ for any $\alpha<\kappa$.
  \end{itemize}
  And {\rm MA} is the assertion that  {\rm MA}$_{\kappa}$ holds for all $\kappa< 2^\omega$.
  
  A partial order $\mc{P}$ has the \emph{countable chain condition} (ccc)  if every uncountable subset of $\mc{P}$ contains two elements with a common lower bound.
   
    Since formulated, {\rm MA} has been a powerful method, to set-theorists and non-set-theorists, of proving consistent results (see e.g., \cite{Fremlin}, \cite{KV},  \cite{TV}, \cite{Todorcevic2000}, \cite{Bagaria}). For example, a straightforward application of {\rm MA} to $(\mc{B}_\mathbb{R}, \subseteq)$, where $\mc{B}_\mathbb{R}=\{(p, q): p<q$ are rational numbers$\}$ is the collection of basic open subsets of $\mathbb{R}$,  shows 
   \begin{itemize}
   \item the \emph{Strong Baire Category Theorem}: the intersection of $<2^\omega$ dense open sets is dense.
   \end{itemize}
   
   In general, consequences of {\rm MA} are strictly weaker than {\rm MA} and follow from restrictions of {\rm MA} to corresponding subclasses of ccc posets.  For example, the above Strong Baire Category Theorem is equivalent to   Martin's axiom restricted to Cohen forcing.
   
The investigation of stronger ccc properties and their corresponding restricted Martin's axiom has a long histroy, e.g., Knaster \cite{Knaster} investigated property $K$ ($K_n$ in Figure 1 for $n=2$)  and  Kunen-Tall \cite{KT} investigated various versions of restricted Martin's axiom  (see also \cite{Juhasz},  \cite{Fremlin}, \cite{TV}, \cite{Todorcevic91}).

   For a property (or a collection) of ccc partial orders $\Psi$ and cardinal $\kappa$, {\rm MA}$_\kappa(\Psi)$ is the following assertion:
     \begin{itemize}
  \item If $\mc{P}$ has property $\Psi$ and $\{D_\alpha: \alpha<\kappa\}$ is a collection of dense subsets of $\mc{P}$, then there is a filter $G\subset \mc{P}$ such that $G\cap D_\alpha\neq \emptyset$ for any $\alpha<\kappa$.
  \end{itemize}
  And {\rm MA}$(\Psi)$ is the assertion that  {\rm MA}$_{\kappa}(\Psi)$ holds for all $\kappa< 2^\omega$.
  
  Restricted Martin's axiom is related to numerous interesting properties. For example, 
  by Stone's Representation Theorem \cite{Stone},
  \begin{itemize}
  \item {\rm MA},  {\rm MA}(powerfully ccc), and {\rm MA}(countable) are equivalent to   Strong Baire Category Theorem for ccc compact spaces,  powerfully ccc compact spaces, and $\mathbb{R}$ respectively;
  \end{itemize}
  by \cite{Bell} and  \cite{MS},
  \begin{itemize}
  \item $\mathrm{MA}(\sigma\text{-centered}) \Longleftrightarrow \mathfrak{t}=2^\omega.$
  \end{itemize}
 The  following figure collects  most frequently used and well-known strong ccc properties.

    \begin{figure}[h]\label{figure1}
\begin{tikzpicture}[auto, node distance=1.2 cm, >=latex]
\node   (countable) {countable};
\node  (centered) [below left of=countable, xshift=-1.5cm] {$\sigma$-centered};
\node  (precaliber) [below right of=countable, xshift=1.5cm] {precaliber $\omega_1$};
\node (linked1) [below of=centered] {$\sigma$-$(n+1)$-linked};
\node (K1) [below of=precaliber] {$K_{n+1}$};
\node (linked) [below of=linked1]  {$\sigma$-$n$-linked};
\node (K) [below of=K1] {$K_n$};
\node (productive) [below left of=K, xshift=-1.5cm] {productively ccc};
\node (powerful) [below of=productive] {powerfully ccc};
\node (ccc) [below of=powerful] {ccc};

\draw [->] (countable) -- (centered);
\draw [->] (countable) -- (precaliber);
\draw [->] (centered) -- (precaliber);
\draw [->] (centered) -- (linked1);
\draw [->] (precaliber) -- (K1);
\draw [->] (linked1) -- (K1);
\draw [->] (linked1) -- (linked);
\draw [->] (K1) -- (K);
 \draw [->] (linked) -- (K);
 \draw [->] (linked) -- (productive);
 \draw [->] (K) -- (productive);
 \draw [->] (productive) -- (powerful);
 \draw [->] (powerful) -- (ccc);

\end{tikzpicture}
\caption{ }
\end{figure}
In above figure, arrow means implication. Implications in Figure 1 are known to be non-reversible (see \cite{Bagaria} for a more detailed survey).  See Definition \ref{definition2} below for definition of properties in above figure.

Clearly, the stronger the property $\Psi$ is, the weaker {\rm MA}($\Psi$) is. The axiom version of above figure is Figure 2.

   \begin{figure}[h]\label{fig2}
\begin{tikzpicture}[auto, node distance=1.2 cm, >=latex]
\node   (MA) {{\rm MA}};
\node (powerful) [below of=MA] {{\rm MA}(powerfully ccc)};
\node (productive) [below  of=powerful] {{\rm MA}(productively ccc)};
\node (K) [below left of=productive, xshift=-1.5cm] {{\rm MA}($K_n$)};
\node (linked) [below right of=productive, xshift=1.5cm] {{\rm MA}($\sigma$-$n$-linked)};
\node (K1) [below of=K] {{\rm MA}($K_{n+1}$)};
\node (linked1) [below of=linked] {{\rm MA}($\sigma$-$(n+1)$-linked)};
\node (precaliber) [below of=K1] {{\rm MA}(precaliber $\omega_1$)};
\node (centered) [below of=linked1] {{\rm MA}($\sigma$-centered)};
\node (countable) [below left of=centered, xshift=-1.5cm] {{\rm MA}(countable)};

\draw [->] (MA) -- (powerful);
\draw [->] (powerful) -- (productive);
\draw [->] (productive) -- (K);
\draw [->] (productive) -- (linked);
\draw [->] (K) -- (linked);
\draw [->] (K) -- (K1);
\draw [->] (linked) -- (linked1);
\draw [->] (K1) -- (linked1);
\draw [->] (K1) -- (precaliber);
 \draw [->] (linked1) -- (centered);
 \draw [->] (precaliber) -- (centered);
 \draw [->] (precaliber) -- (countable);
 \draw [->] (centered) -- (countable);

\end{tikzpicture}
\caption{ }
\end{figure}

 And implications in Figure 2 were known to be non-reversible except possibly for the first implication. In particular, all properties below {\rm MA}(powerfully ccc) in Figure 2 are strictly weaker than \ma. The following natural question was explicitly asked in \cite{Todorcevic91}.
\begin{question}[\cite{Todorcevic91}]\label{q1}
Does {\rm MA}(powerfully ccc) imply {\rm MA}?
\end{question}
 
 Another reason to expect a positive answer to Question \ref{q1} is an attempt to achieve the consistency of {\rm MA} together with some other property (e.g., the consistency of {\rm MA} +$\neg \mathrm{CH}$ + every linear order of size continuum can be embedded into $(\omega^\omega, <^*)$ in \cite{Woodin}). Due to nice properties of a powerfully ccc poset $\pp$, e.g., adding no branches to Suslin trees, being representable  as  a natural ccc partition on $[|\pp|]^2$ (see, e.g.,   \cite[pp. 837]{Todorcevic91}), a positive answer to the following question might be useful in achieving above attempt.
 \begin{question}\label{q2}
 Is there always a powerfully ccc poset to force \ma?
 \end{question}

A natural strategy to distinguish {\rm MA}($\Psi$) from {\rm MA}($\Psi'$) for $\Psi\rightarrow \Psi'$ (in Figure 1 below countable) is to first find a property that can be destroyed by a $\Psi'$ poset but not a $\Psi$ poset (e.g., a poset having property $\Psi'$ but not $\Psi$)
 and then iteratively force with $\Psi$ posets. If  $\Psi$ is preserved in the finite support iteration, then the final model would satisfy {\rm MA}($\Psi$) but not {\rm MA}($\Psi'$). It turns out that productively ccc, $K_n$ and precaliber $\omega_1$ are preserved under finite support iterated forcing; $\sigma$-$n$-linked and $\sigma$-centered are also preserved if the iteration has length $\leq 2^\omega$ (see \cite{Barnett}).  However, powerfully ccc is not preserved in iterated forcing (see Example \ref{exam1} in Section 5).

In order to solve Question \ref{q1}, we analyze the damage that a ccc forcing causes to a strong coloring. Then we introduce an iteration of ccc posets to minimize the damage to the pre-chosen strong coloring (see Definition \ref{varphi} in Section 3). In the iteration process, a ccc poset will either be forced or have an uncountable antichain added to some of its finite power.  In particular, {\rm MA}(powerfully ccc) holds in the final model. On the other hand, the damage to a pre-chosen strong coloring $\pi: [\omega_1]^2\ra 2$ is minimal so that it is possible to preserve ccc of the poset adding uncountable $l$-homogeneous subset by finite approximation  for $l<2$ (see Definition \ref{defn11} in Section 4). Eventually we get a negative answer to Question \ref{q1}.
 \begin{thm}\label{theorem1}
 {\rm MA}(powerfully ccc) does not imply {\rm MA}.
 \end{thm}
    We will use the following notation for  strong colorings introduced in \cite{Galvin} and \cite{Shelah94}.\footnote{A similar property {\rm Pr}$_1(\omega_1, 2, \omega)$   requiring $h$ in {\rm Pr}$_0(\omega_1, 2, \omega)$ to be constant was investigated in \cite{Galvin}.}
    For $2\leq \kappa\leq \omega_1$ and $\theta\leq \omega$, Pr$_0(\omega_1, \kappa, \theta)$ is the following assertion:
\begin{itemize}
  \item  There is a function $\pi:[\omega_1]^2\rightarrow \kappa$ such that whenever we are given $n<\theta$, a collection $\langle a_\alpha: \alpha< \omega_1\rangle$ of pairwise disjoint elements of $[\omega_1]^{n}$ and  $h: n\times n\ra \kappa$,   there are $\alpha< \beta$ such that $\pi( a_\alpha(i), a_\beta(j))=h(i, j)$ for all $i, j<n$.
  \end{itemize}
  In the iteration process, Pr$_0(\omega_1, 2, \omega)$ of $\pi$ will be destroyed while ccc of poset to add 0 (1) homogeneous subset will be preserved.
  
  It   turns out that a negative answer to Question \ref{q1} gives a negative answer to Question \ref{q2}. To see this, suppose $V$ is a model of {\rm MA}(powerfully ccc) + $\neg \mathrm{MA}$ and $\pp$ is a powerfully ccc poset. Then a ccc poset $\mc{Q}$ witnessing the failure of {\rm MA} in $V$ also witnesses the failure of {\rm MA} in the extension over $\pp$. On one hand, $\mc{Q}$ cannot be powerfully ccc in $V$ and hence is not powerfully ccc in $\pp$-extension. On the other hand, $\mc{Q}$ is still ccc in $\pp$-extension since if a powerfully ccc poset can add an uncountable antichain to $\mc{Q}$, then some uncountable antichain already exists in $V$ by {\rm MA}(powerfully ccc).
  \begin{thm}
  Assume {\rm MA}(powerfully ccc) + $\neg\mathrm{MA}$. Then {\rm MA} fails in any powerfully ccc forcing extension.
  \end{thm}
  
  In view of {\rm MA}(powerfully ccc)$\not\ra$ \ma, one may expect a natural stronger restricted Martin's axiom that might imply \ma. And a natural strengthening  of {\rm MA}(powerfully ccc) is {\rm MA}(squarely ccc). A ccc poset is \emph{squarely ccc} if its square is ccc. However, we will modify the proof and show that this attempt fails.
  \begin{thm}
  {\rm MA}(squarely ccc) does not imply \ma.
  \end{thm}
  
  This paper is organized as follows. Section 3 introduces the iteration of ccc posets with minimal damage to a strong coloring. Applying the method, one always abtains a model of {\rm MA}(powerfully ccc). Theorem \ref{theorem1} is proved  in Section 4  where a new property is introduced and preserved in the iteration process to guarantee the failure of \ma. Section 5 shows that Theorem \ref{theorem1} is consistent with continuum arbitrarily large. Section 6 introduces several strengthening of {\rm MA}(powerfully ccc) together with the failure of \ma.

    \section{Preliminary}

      Throughout this paper, $\subset$ and $\supset$ are strict and denote $\subsetneq$ and $\supsetneq$ respectively.     We then  introduce several definitions and notations.     
    
    For a function $f$ and a set $X\subseteq dom(f)$, denote $f[X]=\{f(x): x\in X\}$.
    
    \begin{defn}
  \begin{enumerate}
    \item[(i)] A set of ordinals $X$ is identified with the increasing sequence which enumerates $X$.  In particular, for $\alpha$ less than the order type of $X$ and a subset $I$ of the order type of $X$,   $X(\alpha)$ is the $\alpha$th element  in the increasing enumeration of $X$ and $X[I]=\{X(\beta): \beta\in I\}$.
    \item[(ii)] Suppose   $a$, $b$ are both finite sets of ordinals but neither is an ordinal. 
    Say $a<b$ if $\max a<\min b$. 
    \item[(iii)] A set $\ms{A}\subset [Ord]^{<\omega}$ is \emph{non-overlapping} if for every $a\neq b$ in $\ms{A}$, either $a< b$ or $ b< a$.
    \item[(iv)] For an uncountable non-overlapping family $\msa\subset [\omega_1]^n$, use $N_\msa$ to denote this $n$, i.e., $\msa\subset [\omega_1]^{N_\msa}$.
    \item[(v)] For an uncountable non-overlapping family $\msa\subset [\omega_1]^{N_\msa}$ and $i<N_\msa$, denote $\msa_i=\{a(i): a\in \msa\}$.
  \end{enumerate}
\end{defn}

The following notions of posets are standard (see, e.g.,  \cite{Barnett}, \cite{Bagaria}).
\begin{defn}\label{definition2}
Suppose $\pp$ is a partial order and $2\leq n<\omega$.
\begin{itemize}
\item $\mc{P}$ is \emph{powerfully ccc} if all finite powers of $\pp$ are ccc.
\item $\pp$ is \emph{productively ccc} if its product with every ccc partial order is ccc.
\item $\pp$ has \emph{property $K_n$} if every uncountable subset of $\pp$ has an uncountable subset that is $n$-linked.   A subset $X$ of $\pp$ is \emph{$n$-linked} if every $n$-element subset of $X$ has a common lower bound.  
\item $\pp$ is \emph{$\sigma$-$n$-linked} if $\pp$ is a countable union of $n$-linked subsets.  
\item $\pp$ has \emph{precaliber $\omega_1$} if every uncountable subset of $\pp$ has an uncountable centered subset. A subset $X$ of $\pp$ is \emph{centered} if every finite subset of $X$ has a common lower bound.
\item $\pp$ is \emph{$\sigma$-centered} if $\pp$ is a countable union of centered subsets.
\end{itemize}
In above concepts, we omit $n$ if $n=2$.
\end{defn}

  \section{Iteration with minimal damage to a strong coloring}
  
  We start from a model of  $2^{\omega_1}=\omega_2$ in which there is 
 a coloring $\pi: [\omega_1]^2\ra 2$ witnessing Pr$_0(\omega_1, 2, \omega)$. 
   
  In order to force a model of \map, we have to destroy property Pr$_0(\omega_1, 2, \omega)$ of $\pi$. We first investigate the influence on $\pi$ by a poset adding an uncountable family $\subset [\omega_1]^n$ with pre-described pattern.

  \begin{defn}\label{defn2}
  For $1\leq n<\omega$, an uncountable non-overlapping family $\ms{A}\subset [\omega_1]^n$  and a non-empty collection $H$ of functions from $n\times n$ to $2$, let $\mc{P}_{\ms{A}, H}$ be the poset consisting of $p\in [\ms{A}]^{<\omega}$ such that
  \begin{itemize}
  \item for every $a< b$ in $p$, there is $h\in H$ satisfying $\pi(a(i), b(j))=h(i,j)$ whenever $i, j<n$.
  \end{itemize}
  $p\leq q$ if $p\supseteq q$.
  \end{defn}
  Note that if $\pi$ satisfies Pr$_0(\omega_1, 2, \omega)$, then $\ms{P}_{\ms{A}, H}$ is ccc. Throughout this paper, we assume, omitting a countable part of $\msa$ if necessary, that a $\pah$-generic filter is uncountable.
  
  Before introducing the structure associated to the iteration, we first investigate the influence of $\ms{P}_{\ms{A}, H}$ on color $\pi$. But first, we introduce some notions to simplify the statements. 
  
  For finite sets of ordinals $a, b, c, d$,
  \begin{itemize}
  \item for functions $f, g$ with domain $a\times b$, $c\times d$ respectively, say $f$ is \emph{order isomorphic} to $g$ if $|a|=|c|, |b|=|d|$ and $f(a(i), b(j))=g(c(i), d(j))$ whenever $i<|a|, j<|b|$. 
  \end{itemize}
  \begin{defn}
 For uncountable non-overlapping family $\ms{A}\subset [\omega_1]^n$ and $f: n\times n\ra 2$, say \emph{$f$ can be realized in $\ms{A}$} if for some $a< b$ in $\ms{A}$, $f$ is order isomorphic to $\pi\up_{a\times b}$.\footnote{Here we identify $\{\{a(i), b(j)\}: i, j<n\}$ with $a\times b$.}
  \end{defn}

  \begin{lem}\label{lem1}
  Suppose $\pi$ witnesses {\rm Pr}$_0(\omega_1, 2, \omega)$ and  $n ,\ms{A}, H, \ms{P}_{\ms{A}, H}$  are as above. Suppose $G$ is $\ms{P}_{\ms{A}, H}$-generic and for some $m$, $\ms{B}\subset [\omega_1]^m$ is an uncountable non-overlapping family in $V[G]$  such that for some $k$ and pairwise disjoint $\{I_i\in [m]^n: i<k\}$, $$[b]^n\cap \bigcup G=\{b[I_i]: i<k\}\text{ and }b(j)\notin \bigcup\bigcup G$$
   whenever $b\in \ms{B}$, $j\in m\setminus \bigcup_{i<k} I_i$. Then a function $f: m\times m\ra 2$ can be realized in $\ms{B}$ iff for every $i, j<k$, $f\up_{I_i\times I_j}$ is order isomorphic to some $h\in H$.
  \end{lem}
  \begin{proof}
  The only if part follows from the definition of $\ms{P}_{\ms{A}, H}$ and the properties of $\msb$.  
  To prove the if part, we fix a $\pah$ name $\dot{\msb}$ of $\msb$ and a condition $p\in G$ which forces $\dot{\msb}$ to satisfy properties of $\msb$.
  
  For each $\alpha<\omega_1$, find $p_\alpha\leq p$ and $b_\alpha\in [\omega_1]^m$ such that
  \begin{enumerate}
  \item $p_\alpha\Vdash b_\alpha\in \dot{\msb}\cap [\omega_1\setminus \alpha]^m$. Extending $p_\alpha$ if necessary, assume $[b_\alpha]^n\cap  p_\alpha=\{b_\alpha[I_i]: i<k\}$.
  \end{enumerate}
  Find uncountable $\Gamma\subset \omega_1$ such that
  \begin{enumerate}\setcounter{enumi}{1}
  \item $\{p_\alpha: \alpha\in \Gamma\}$ forms a $\Delta$-system with root $r$ and each $p_\alpha\setminus r$ has size $k^*$;
  \item $\{c_\alpha: \alpha\in \Gamma\}$ is non-overlapping where $c_\alpha=b_\alpha\cup \bigcup (p_\alpha\setminus r)$;
  \item for some $m^*$, $J\in [m^*]^m$ and $\{J_i\in [m^*]^n: i<k^*\}$, $|c_\alpha|=m^*$, $b_\alpha=c_\alpha[J]$ and $p_\alpha\setminus r=\{c_\alpha[J_i]: i<k^*\}$ whenever $\alpha\in \Gamma$.
  \end{enumerate}
  Note that $J_i$'s are pairwise disjoint. By properties of $\msb$ and the fact that $p$ forces $\dot{\msb}$ to satisfy properties of $\msb$, each $J_i$ is either contained in $J$ or disjoint from $J$. Hence there exists $g: m^*\times m^*\ra 2$ such that
  \begin{enumerate}\setcounter{enumi}{4}
  \item $g\up_{J\times J}$ is order isomorphic to $f$;
  \item for $i, j<k^*$ with $J_i\cup J_j\not\subseteq J$, $g\up_{J_i\times J_j}$ is order isomorphic to some $h\in H$.
  \end{enumerate}
  Note that by (5) and our choice of $f$,  $g\up_{J_i\times J_j}$ is order isomorphic to some $h\in H$ if $J_i\cup J_j\subseteq J$. Together with (6), we conclude that $g\up_{J_i\times J_j}$ is order isomorphic to some $h\in H$ for all $i, j<k^*$.
  
  Applying Pr$_0(\omega_1, 2, \omega)$, we get $\alpha<\beta$ in $\Gamma$ such that $\pi\up_{c_\alpha\times c_\beta}$ is order isomorphic to $g$. Then it is straightforward to check that $p_\alpha\cup p_\beta$ is a condition extending both $p_\alpha$ and $p_\beta$. Moreover, $p_\alpha\cup p_\beta$ forces that $b_\alpha$ and $b_\beta$ witness that $f$ is realized in $\dot{\msb}$.
  
  Now a standard density argument shows that in $V[G]$, $f$ can be realized in $\msb$.
  \end{proof}
  
  For $\pi$, $n ,\ms{A}, H, \ms{P}_{\ms{A}, H}$ and $G$    as in Lemma \ref{lem1}, the added uncountable family $\bigcup G$ will have the following pattern.
  \begin{itemize}
  \item For $a< b$ in $\bigcup G$, $\pi\up_{a\times b}$ is order isomorphic to some $h\in H$.
  \end{itemize}
  In particular, Pr$_0(\omega_1, 2, \omega)$, which asserts that every function from $m\times m$ to 2 can be realized in every uncountable non-overlapping family $\subset [\omega_1]^m$, fails in $V[G]$ (assuming that $H$ is not the collection of all functions from $n\times n$ to 2). But on the other hand, Lemma \ref{lem1} indicates that for every appropriate $\msb\subset [\omega_1]^m$, every function not contradicting above pattern can be realized in $\msb$.
  
  But during the iteration process, more generic families $\subset [\omega_1]^{<\omega}$ will be added. In order to have a clear description on what patterns can be realized, we need the generically added families to form a ``nice'' structure.
  
    \begin{defn}\label{c}
    Let $\varphi_0(\mc{C})$ be the assertion that the following statements hold.
      \begin{itemize}
  \item[(C1)] Every $\msa\in \mc{C}$ is an uncountable non-overlapping family such that $\msa\subset [\omega_1]^{n}$ for some $1\leq n<\omega$.\footnote{Recall that this $n$ is denoted by $N_\msa$.}
  \item[(C2)] For $\msa, \msb$ in $\mc{C}$, if $\msb_i\subseteq \msa_j$ for some $i<N_\msb$ and $ j<N_\msa$, then there are $k$ and $\{I_l\in [N_\msb]^{N_\msa}: l<k\}$ such that   
  $$[b]^{N_\msa}\cap \msa=\{b[I_l]: l<k\}\text{ and }b(j')\notin \bigcup \msa$$
   whenever  $b\in \msb$ and $j'\in N_\msb\setminus \bigcup_{l<k} I_l$.  
  \end{itemize}
        \end{defn}
    \textbf{Remark.} If $\msb_i=  \msa_j$ for some $\msa, \msb$ in $\mc{C}$, then applying (C2) to $\msb_i\subseteq \msa_j$ and $\msa_j\subseteq \msb_i$, we conclude that $\msa=\msb$.
        
        Here $\mc{C}$ is used to collect all generically added non-overlapping families $\bigcup G$'s.  For example, if our first ccc poset is $\ms{P}_{\ms{A}, H}$ as in Lemma \ref{lem1}, our $\mc{C}$ in the forcing extension, denoted by $\mc{C}^1$,  will be $\{\bigcup G\}$ where $G$ is $\ms{P}_{\ms{A}, H}$-generic.\footnote{The superscript in $\mc{C}^1$ is used to record the length of iteration.}

        Of course we need to record the patterns that every $\msa\in \mc{C}$ has to follow.
  \begin{defn}\label{cr}
     Let $\varphi_0(\pi, \mc{C}, \mc{R})$ be the assertion that $\pi: [\omega_1]^2\ra 2$ is a mapping and $\varphi_0(\mc{C})$ together with the following statements hold.
      \begin{itemize}
       \item[(R1)] $\mc{R}$ is a map with domain $\mc{C}$ such that  for every $\msa\in \mc{C}$, $\mc{R}(\msa)$ is a collection of functions from $N_\msa\times N_\msa$ to 2.
       \item[(R2)] For every $\msa, \msb$ in $\mc{C}$ with $\msb_i\subset \msa_j$ for some $i<N_\msb, j<N_\msa$, if $k$ and $\{I_i\in [N_\msb]^{N_\msa}: i<k\}$ satisfy that   $[b]^{N_\msa}\cap \msa=\{b[I_i]: i<k\}$ for every $b\in \msb$, then $h\up_{I_i\times I_j}$ is order isomorphic to some $g\in \mc{R}(\msa)$ whenever $h\in \mc{R}(\msb)$ and $i<k$, $j<k$.
  \item[(R3)] For every $\msa\in \mc{C}$ and $a< b$ in $\msa$, $\pi\up_{a\times b}$ is order isomorphic to some $h$ in $\mc{R}(\msa)$.  
    \end{itemize}
        \end{defn}
        
      So we use $\mc{R}$ to collect potential patterns that (uncountable subsets of) $\msa\in \mc{C}$ can realize. For example, if our first ccc poset is $\ms{P}_{\ms{A}, H}$ as in Lemma \ref{lem1}, our $\mc{R}$ in the forcing extension, denoted by $\mc{R}^1$, will be $\mc{R}^1(\bigcup G)=H$.
      
      The requirement (R2) is necessary according to Lemma \ref{lem1}. It simply says that if a pattern can be realized in a later stage (in $\msb$), it cannot contradict requirements in any earlier stage (in $\msa$).
        
  We actually need more requirements on $\mc{C}$ if the iteration gets longer. To describe the requirements, we will record coordinate-wise collections of each element of $\mc{C}$ (see (T2) below).
   \begin{defn}\label{crt}
    Let $\varphi_0(\pi, \mc{C}, \mc{R}, \mc{T})$ be the assertion that $\varphi_0(\pi, \mc{C}, \mc{R})$ together with the following statements hold.
      \begin{itemize}
        \item[(T1)] $(\mc{T}, \supset)$ is a downward tree of height $\leq \omega$.
  \item[(T2)] $\mc{T}_0=\{\omega_1\}$ and for every $A$, $A\in \mc{T}\setminus \mc{T}_0$ iff  $A=\msa_i$ for some $\msa\in \mc{C}$ and $i<N_\msa$.\footnote{Another option for notations is to require $[\omega_1]^1\in \mc{C}$, $\mc{R}([\omega_1]^1)=\{\{((0,0),0)\}, \{((0,0),1)\}\}$ and $A\in \mct$ iff $a=\msa_i$ for some $\msa\in \mc{C}$ and $i<N_\msa$.}
  \item[(T3)] For every incomparable $A, B$ in $\mc{T}$, $|A\cap B|\leq \omega$.
  \item[(T4)] For every $A\supset B$ in $\mc{T}$ and every $\alpha<\omega_1$, $A(\alpha)<B(\alpha)$.
    \end{itemize}
        \end{defn}
        (T3) is used to reduce the damage among elements of $\mc{C}$. For example, if $|\msa_i\cap \msb_j|\leq \omega$ for every $i<N_\msa, j<N_\msb$ where $\msa, \msb\in \mc{C}$, then the requirements on $\msa$ will not affect $\msb$. 
        
        Before analyzing the role that (T4) plays, we first introduce one more notation to record the equivalence relation on $[\omega_1]^{\omega_1}$ induced by $\mc{C}$.
          
  \begin{defn}\label{cter}
Let $\varphi_0(\pi, \mc{C}, \mc{R}, \mc{T}, \mathbf{E})$ be the assertion that $\varphi_0(\pi, \mc{C}, \mc{R}, \mc{T})$ together with the following statement hold.
  \begin{itemize}
  \item[(E1)] $\mathbf{E}$ is an equivalence relation on $[\omega_1]^{\omega_1}$ defined by $A\mathbf{E}B$  iff $A=B$ or for some   $\msa\in \mc{C}$, $i, j<N_\msa$ and $\Gamma\in [\omega_1]^{\omega_1}$, $A=\msa_i[\Gamma]$ and $B=\msa_j[\Gamma]$.
  \end{itemize}
  \end{defn}
  We   check that $\mathbf{E}$ is an equivalence relation and hence the above definition is well-defined.
   
  \begin{lem}
  Assume $\varphi_0(\pi, \mc{C}, \mc{R}, \mc{T})$. Then $\mathbf{E}$ defined by (E1) is an equivalence relation.
  \end{lem}
  \begin{proof}
  Only transitivity needs a proof. So suppose $A\mathbf{E}B$ and $B\mathbf{E}C$. The case that $A=B$ or $B=C$ is trivial. So assume 
  \begin{enumerate}
  \item for some $\msa\in \mc{C}$, $i, j<N_\msa$  and $\Gamma\in [\omega_1]^{\omega_1}$, $A=\msa_i[\Gamma]$ and $B=\msa_j[\Gamma]$;
  \item for some $\msb\in \mc{C}$, $i', j'<N_\msb$  and $\Gamma'\in [\omega_1]^{\omega_1}$, $B=\msb_{i'}[\Gamma']$ and $C=\msb_{j'}[\Gamma']$.
  \end{enumerate}
  Then $\msa_j\cap \msb_{i'}$ is uncountable. By (T3), either $\msa_j\subseteq \msb_{i'}$ or $\msb_{i'}\subseteq \msa_j$.
  
  We first assume that $\msb_{i'}\subset \msa_{j}$. Then by (C2), there is $I\in [N_\msb]^{N_\msa}$ such that $i'\in I$ and $b[I]\in \msa$ whenever $b\in \msb$. 
  
  Recall $B=\msa_j[\Gamma]=\msb_{i'}[\Gamma']$. Arbitrarily choose $\alpha<\omega_1$. Suppose $a$ is the $\Gamma(\alpha)$th element (under $<$) of $\msa$ and $b$ is the $\Gamma'(\alpha)$th element of $\msb$. Then 
  $$a(j)=b(i')= B(\alpha)\in a\cap b[I].$$ 
 By the fact that $b[I]\in \msa$ and $\msa$ is non-overlapping,
  \begin{itemize}
  \item[(3)] $a=b[I]$, $i'=I(j)$; 
  \item[(4)] $A(\alpha)=a(i)=b(I(i))$ is the $\Gamma(\alpha)$th element of $\msa_i$ and $\Gamma'(\alpha)$th element of $\msb_{I(i)}$.
  \end{itemize}
  
  So $A=\msa_i[\Gamma]=\msb_{I(i)}[\Gamma']$. Together with (2), we conclude $A\mathbf{E}C$.
  
  The case $\msa_j\subset\msb_{i'}$ follows from an argument similar to the above one.  Just note that $C\bfe B$ and $B\bfe A$.
  
  The case $\msa_j=\msb_{i'}$ follows from the definition (E1). Just note by the remark after Definition \ref{c}, $\msa=\msb$ and hence $j=i'$, $\Gamma=\Gamma'$.
    \end{proof}
We then observe the following property of $\mathbf{E}$ describing one equivalence class.
    \begin{lem}\label{lem3}
    Assume $\varphi_0(\pi, \mc{C}, \mc{R}, \mc{T}, \mathbf{E})$ and $A\mathbf{E}B$.
 Suppose for some $\msa\in \mc{C}$, $i<N_\msa$ and $\Gamma\in [\omega_1]^{\omega_1}$, $A=\msa_i[\Gamma]$ and $A\not\subseteq A'$ for any $A'\subset \msa_i$ in $\mc{T}$.  Then $B=\msa_j[\Gamma]$ for some $j<N_\msa$. Hence, $[A]_\mathbf{E}=\{\msa_k[\Gamma]: k<N_\msa\}$  where $[A]_\mathbf{E}$ is the $\mathbf{E}$-equivalence class of $A$.
    \end{lem}
    \begin{proof}
 Without loss of generality, assume $A\neq B$. So for some $\msb\in \mc{C}$, $i', j'<N_\msb$ and $\Sigma\in [\omega_1]^{\omega_1}$, $A=\msb_{i'}[\Sigma]$ and $B=\msb_{j'}[\Sigma]$. 
    
     Then $\msa_i\cap \msb_{i'}$ is uncountable. By (T3), either $\msa_i\subseteq \msb_{i'}$ or $\msb_{i'}\subseteq \msa_i$. If $\msa_i=\msb_{i'}$, then by (C2), $\msa=\msb$ and we are done. So assume $\msa_i\neq\msb_{i'}$. Since $A$ is not contained in any $\supset$-extension of $\msa_i$ in $\mc{T}$, we conclude that $\msa_i\subset \msb_{i'}$.
  
  Then by (C2), there is $I\in [N_\msa]^{N_\msb}$ such that $i\in I$ and $a[I]\in \msb$ whenever $a\in \msa$.   
  Recall $A=\msa_i[\Gamma]=\msb_{i'}[\Sigma]$. Then $i=I(i')$ and $\msa_{I(i')}[\Gamma]=\msb_{i'}[\Sigma]$. Hence $\msa_{I(j')}[\Gamma]=\msb_{j'}[\Sigma]=B$.
    \end{proof}
  
  $\mc{T}$ and $\mathbf{E}$ are induced from $\mc{C}$. But (T1)-(T4) add some restrictions on $\mc{C}$ while (E1) is automatically satisfied.
  
  Note that for $A\supset B$ in $\mc{T}$, $[A]_\mathbf{E}$ has size less than or equal to $[B]_\mathbf{E}$. We do not want any equivalence class to be infinite. So one reason to require (T4) is to make sure  $ht(\mc{T})\leq \omega$ and every $\mathbf{E}$-equivalence class is finite.
  
  \begin{lem}\label{lem4}
  Assume $\varphi_0(\pi, \mc{C}, \mc{R}, \mc{T}, \mathbf{E})$. For a sequence $\langle A_n: n<\omega\rangle$ of $\supset$-chain in $\mc{T}$, $\bigcap_{n<\omega} A_n=\emptyset$.
  \end{lem}
  \begin{proof}
  Suppose towards a contradiction that $\xi\in \bigcap_{n<\omega} A_n$. For each $n$, choose $\alpha_n$ such that $\xi=A_n(\alpha_n)$. By (T4) and the fact that $A_n\supset A_{n+1}$, 
  $$A_{n+1}(\alpha_{n+1})=A_n(\alpha_n)<A_{n+1}(\alpha_n)\text{ and hence }\alpha_{n+1}<\alpha_n.$$ 
  This induces an infinite decreasing sequence of ordinals $\alpha_0>\alpha_1>\cdot\cdot\cdot$. A contradiction.
  \end{proof}
  
  \begin{lem}\label{lem5}
  Assume $\varphi_0(\pi, \mc{C}, \mc{R}, \mc{T}, \mathbf{E})$.  For every $A\in [\omega_1]^{\omega_1}$, 
  \begin{enumerate}
  \item $[A]_\bfe$ is finite;
  \item there exists unique $A^*\in \mc{T}$ such that $A\subseteq A^*$ and $A\not\subseteq A'$ for any $A'\subset A^*$ in $\mc{T}$.
  \end{enumerate}
  \end{lem}
  \begin{proof}
  Let $\mc{A}=\{\msa\in \mc{C}: \exists i<N_\msa \ A\subseteq\msa_i\}$. If $\mc{A}=\emptyset$, then $[A]_\bfe=\{A\}$ is finite and $A^*=\omega_1$. So assume $\mc{A}\neq\emptyset$.
  
  For each $\msa\in \mc{A}$, let $i_\msa$ be such that $A\subseteq \msa_{i_\msa}$. Then $\{\msa_{i_\msa}: \msa\in \mc{A}\}$ is a $\supset$-chain by (T3) and finite by Lemma \ref{lem4}.  
  Let $\msa\in \mc{A}$ be such that $\msa_{i_\msa}$ is the last element in the $\supset$-chain.
  
  Then $A\not\subseteq A'$ for any $A'\subset \msa_{i_\msa}$ in $\mc{T}$. By Lemma \ref{lem3}, $[A]_\bfe=\{\msa_k[\Gamma]: k<N_\msa\}$ is finite where $\Gamma$ is such that $A=\msa_{i_\msa}[\Gamma]$. And uniqueness of $A^*=\msa_{i_\msa}$ follows from (T3).
  \end{proof}
  
  We record the unique $A^*$ for each $A$ as indicated in above lemma.
  \begin{defn}
  Assume $\varphi_0(\pi, \mc{C}, \mc{R}, \mc{T}, \mathbf{E})$.  Define $\Phi_\mc{T}: [\omega_1]^{\omega_1}\ra \mc{T}$ by
  $$\Phi_\mc{T}(A) \text{ is the unique } A^*\in \mc{T} \text{ such that  $A\subseteq A^*$ and $A\not\subseteq A'$ for any $A'\subset A^*$ in $\mc{T}$}.$$
  \end{defn}
  Note that $\Phi_\mc{T}(A)=A$ iff $A\in \mc{T}$. The proof of Lemma \ref{lem5} shows the following.
  
  \begin{lem}\label{lem6}
  Assume $\varphi_0(\pi, \mc{C}, \mc{R}, \mc{T}, \mathbf{E})$.  For every $A\in [\omega_1]^{\omega_1}$,   either $[A]_\bfe=\{A\}$ or
  $$[A]_\bfe=\{\msa_i[\Gamma]: i<N_{\msa}\}$$
   where $\msa\in \mc{C}$ and $\Gamma\in [\omega_1]^{\omega_1}$ are such that $\Phi_\mc{T}(A)=\msa_i$ for some $i$ and $A=\Phi_\mc{T}(A)[\Gamma]$.
  \end{lem}
  
  Next, we observe the following properties.
  \begin{lem}\label{lem7}
    Assume $\varphi_0(\pi, \mc{C}, \mc{R}, \mc{T}, \mathbf{E})$ and $A\mathbf{E}B$.   Then the following statements hold.
    \begin{enumerate}
    \item Suppose $A\neq B$. Then for some $\msa\in \mc{C}$ and $i, j<N_\msa$, $\msa_i=\Phi_\mct(A)$ and $\msa_j=\Phi_\mct(B)$. In particular, $\Phi_\mct(A)\bfe \Phi_\mct(B)$.
     \item $|A\cap A'|\leq \omega$ for every $A'\subset \Phi_\mc{T}(A)$ in $\mc{T}$ iff $\Phit(A'')=\Phit(A)$ for every $A''\in [A]^{\omega_1}$.
    \item $|A\cap A'|\leq \omega$ for every $A'\subset \Phi_\mc{T}(A)$ in $\mc{T}$ iff $|B\cap B'|\leq \omega$ for every $B'\subset \Phi_\mc{T}(B)$ in $\mc{T}$.
       \item $\Phi_\mct(A)$ is an end node in $\mct$ iff $\Phi_\mct(B)$ is an end node in $\mct$.
    \end{enumerate}
      \end{lem}
  \begin{proof}
  (1) Fix $\msa, \msb$ in $\mc{C}$ and $i<N_\msa, j<N_\msb$ such that $\Phi_\mct(A)=\msa_i$ and $\Phi_\mct(B)=\msb_j$. 
  
  By Lemma \ref{lem6}, $A\subseteq \msb_{i'}$ for some $i'<N_\msb$. By (T3) and the fact that $A\subseteq \Phi_\mct(A)=\msa_i$, either $\msa_i\subseteq \msb_{i'}$ or $\msb_{i'}\subset \msa_i$. Then by definition of $\Phi_\mct(A)$, $\msa_i\subseteq \msb_{i'}$.
  
  Similarly, $\msb_j\subseteq \msa_{j'}$ for some $j'<N_\msa$. Then by (C2), $\msa=\msb$.\medskip
  
  (2)   ``$\Leftarrow$''. If $|A\cap A'|= \omega_1$ for some $A'\subset \Phi_\mc{T}(A)$ in $\mc{T}$, then 
  $$\Phit(A)=\Phit(A\cap A')\subseteq A'\subset \Phit(A)$$
   which cannot happen.
  
  ``$\Rightarrow$''.  On one hand, $A''\subseteq A\subseteq \Phit(A)$. On the other hand, for every $A'\subset \Phi_\mct(A)$ in $\mc{T}$, $A''\cap A'\subseteq A\cap A'$ is countable and hence $A''\not\subseteq A'$.\medskip
  
  (3)  Clearly, we may assume $A\neq B$. Then by (1) and Lemma \ref{lem6}, 
  $$A=\msa_i[\Gamma], B=\msa_j[\Gamma], \Phit(A)=\msa_i \text{ and } \Phit(B)=\msa_j$$
   for some $\msa\in \mc{C}$, $i, j<N_\msa$ and $\Gamma\in [\omega_1]^{\omega_1}$. Note by (C2), for every $\Sigma\in [\omega_1]^{\omega_1}$, $\msa_i[\Sigma]\in \mct$ iff $\msa_j[\Sigma]\in \mct$.
   This implies  the equivalence.\medskip

  (4) By (1), $\Phi_\mct(A)\bfe \Phi_\mct(B)$. Then applying (3) to $\Phi_\mct(A)$ and $ \Phi_\mct(B)$, we get (4).
  Just note that $\Phi_\mct(A)$ is an end node in $\mct$ iff $|\Phi_\mct(A)\cap A'|\leq \omega$ for every $A'\subset \Phi_\mc{T}(A)$ in $\mc{T}$.
    \end{proof}
  
  To illustrate the possible patterns that should be reserved by $(\mc{C}, \mc{R}, \mc{T}, \mathbf{E})$  as in Lemma \ref{lem1}, we need to first introduce   families that are ``appropriate'' for $\mc{C}$.
  
  \begin{defn}\label{candidate}
  Assume $\varphi_0(\pi, \mc{C}, \mc{R}, \mc{T}, \mathbf{E})$. Say $\msa$ is \emph{a candidate for $\mc{C}$} (or say \emph{a $\mc{C}$-candidate}) if for some $n<\omega$, $\msa\subset [\omega_1]^{n}$ is uncountable non-overlapping  and 
  \begin{itemize}
  \item[(i)] $\{\msa_i: i<n\}$ is $\bfe$-invariant, i.e., $[\msa_j]_\bfe\subseteq\{\msa_i: i<n\}$ for every $j<n$; 
  \item[(ii)] for every $i<n$ and every $A\subset \Phi_\mct(\msa_i)$ in $\mc{T}$, $|\msa_{i}\cap A|\leq \omega$;
  \item[(iii)] for every $i<n$ and every $\alpha<\omega_1$, $\Phi_\mct(\msa_i)(\alpha)< \msa_{i}(\alpha)$;
  \item[(iv)] for every $i, j<n$ and every $A\supseteq \Phit(\msa_j)$ in $\mct$, either $\msa_i\subset A$ or $\msa_i\cap A=\emptyset$.
  \end{itemize}
  \end{defn}
  
  So  $\mc{C}$-candidates   are candidates that might be added into $\mc{C}$. More precisely, if $\msa$ is a candidate for $\mc{C}$, then we may add $\msa$ to $\mc{C}$ and add $\msa_i$'s to $\mc{T}$ as end nodes while (C1)-(C2) and (T1)-(T4) are all satisfied. In other words, we only need to take care of the $\mc{R}$ part.
  Of course, we will generically add a subset of a candidate into $\mc{C}$. But the preparation can be done in advance.
  
  We observe some basic facts of   $\mc{C}$-candidates.
  \begin{lem}\label{lem8}
    Assume $\varphi_0(\pi, \mc{C}, \mc{R}, \mc{T}, \mathbf{E})$. Suppose $\msa$ is a candidate for $\mc{C}$.
    \begin{enumerate}
    \item For every $A\in \mct$ and $i<N_\msa$, $|\msa_i \cap A|=\omega_1$ iff $\msa_i\subset A$ iff $\Phit(\msa_i)\subseteq A$.
    \item For every $\msb\in [\msa]^{\omega_1}$, $\msb$ is a candidate for $\mc{C}$. 
    \item There are $k<\omega$, a sequence $\langle \msa^i\in \mc{C}: i<k\rangle $ and a partition $N_\msa=I\cup\bigcup_{i<k} I_i$ such that 
    \begin{enumerate}
    \item $\Phit(\msa_j)=\omega_1$ for $j\in I$;
    \item $\Phit(\msa_{I_i(j)})=\msa^i_j$ for $i<k$ and  $j<|I_i|=N_{\msa^i}$;
    \item $a[I_i]\in \msa^i$ for $a\in \msa$ and $i<k$.
        \end{enumerate}
        \item for $\msb\in \mc{C}$ and $J\in [N_\msa]^{N_\msb}$, if $\{a\in \msa: a[J]\in \msb\}$ is uncountable, then $a[J]\in \msb$ whenever $a\in \msa$.
    \end{enumerate}   
  \end{lem}
  \begin{proof}
  (1) It suffices to prove that $|\msa_i \cap A|=\omega_1$ implies $\Phit(\msa_i)\subseteq A$. Assume $|\msa_i \cap A|=\omega_1$. By (T3), $A$ is comparable with $\Phit(\msa_i)$. By Definition \ref{candidate} (ii) for $\msa$, $\Phit(\msa_i)\subseteq A$. \medskip
  
(2) By Lemma \ref{lem7} (2), for every $i<N_\msa$, $\Phit(\msb_i)=\Phit(\msa_i)$.
  
  Now Definition \ref{candidate} (i) for $\msb$ follows from Lemma \ref{lem6} and Definition \ref{candidate} (i) for $\msa$. Definition \ref{candidate} (ii)-(iv) for $\msb$ follow from $\Phit(\msb_i)=\Phit(\msa_i)$ and $\msb\subseteq  \msa$.\medskip
  
  (3) Let $I=\{j<N_\msa: \Phit(\msa_j)=\omega_1\}$ and $k=|\{\msa_j: j\in N_\msa\setminus I\}/\bfe|$ be the number of $\bfe$-equivalence classes. Now $N_\msa=I\cup\bigcup_{i<k} I_i$ is the partition such that $\{\msa_j: j\in I_i\}$ is one equivalence class for each $i<k$.
  
  For $i<k$, by Lemma \ref{lem6} and Lemma \ref{lem7}, there are $\msa^i\in \mc{C}$ and $\Gamma\in[\omega_1]^{\omega_1}$ such that $\{\msa_j: j\in I_i\}=\{\msa^i_{j}[\Gamma]: j<N_{\msa^i}\}$ and $\Phit(\msa_{I_i(j)})=\msa^i_{j}$ for $j<|I_i|$.
  Clearly, $a[I_i]\in \msa^i$ for $a\in \msa$ and $i<k$.
  
(4)  Fix $\msb\in \mc{C}$ and $J\in [N_\msa]^{N_\msb}$ such that $\{a\in \msa: a[J]\in \msb\}$ is uncountable. Let $k$,  $\langle \msa^i\in \mc{C}: i<k\rangle $ and $N_\msa=I\cup\bigcup_{i<k} I_i$ be guaranteed by (3). By (T3), $\Phit(\msa_{J(0)})$ is comparable with $\msb_0$. By Definition \ref{candidate} (ii), $\Phit(\msa_{J(0)})\subseteq \msb_0$. So  $\msa^i_j\subseteq \msb_0$ for   $i, j$ with $I_i(j)=J(0)$. By (C2), for some $J'$ with $j\in J'$, $a'[J']\in \msb$ whenever $a'\in \msa^i$. Clearly, $J=I_i[J']$ and $a[J]\in \msb$ whenever $a\in \msa$.
    \end{proof}
  \textbf{Remark.} In above lemma, repetition is allowed in $\langle \msa^i: i<k\rangle$.
  
  We would like to find, for each $A\in [\omega_1]^{\omega_1}$, an uncountable subset whose $\bfe$-closure is a coordinate-wise collection of a $\mc{C}$-candidate. The only difficulty is the satisfaction of Definition \ref{candidate} (ii).  
  We will see below that under some condition, Definition \ref{candidate} (ii) can be satisfied.
  \begin{lem}\label{lem candidate}
   Assume $\varphi_0(\pi, \mc{C}, \mc{R}, \mc{T}, \mathbf{E})$ and $|\mct|\leq \omega_1$. For every $A\in [\omega_1]^{\omega_1}$, there exists $B\in [A]^{\omega_1}$ such that
   $$\text{ for every } B'\subset \Phi_\mct(B) \text{ in } \mct, |B\cap B'|\leq \omega.$$
  \end{lem}
  \begin{proof}
  First assume that for some $C\in \mct$, 
  \begin{enumerate}
  \item $|A\cap C|=\omega_1$ and for every countable subset $\mct'$ of $\{C'\in \mct: C'\subset C\}$, $|A\cap C\setminus \bigcup \mct'|=\omega_1$.
  \end{enumerate}
  
  Then enumerate $\{C'\in \mct: C'\subset C\}$ as $\{C_\alpha: \alpha<\kappa\}$ for some $\kappa$. Since $|\mct|\leq \omega_1$, we may assume $\kappa\leq \omega_1$.
  
  If $\kappa<\omega_1$, then take $B=A\cap C\setminus \bigcup_{\alpha<\kappa} C_\alpha$.
  
  If $\kappa=\omega_1$, then inductively choose $x_\alpha\in A\cap C\setminus \bigcup_{\xi<\alpha} C_\xi$ that is greater than $\sup \{x_\xi: \xi<\alpha\}$. Take $B=\{x_\alpha: \alpha<\omega_1\}$.
  
  In each case, it is straightforward to check that $\Phit(B)=C$ and $B$ is as desired.\medskip
  
  Now it suffices to prove that there exists $C\in \mct$ satisfying (1). Suppose towards a contradiction that no such $C$ exists. Then, by (T3),
  $$\mct'=\{C\in \mct: |A\cap C|=\omega_1\}$$
   is a countably branching subtree, i.e., for every $C\in \mct'$, $C$ has countably many immediate successors in $\mct'$.  To see this, for every $C\in \mct'$, let $succ_{\mct'}(C)$ be the collection of immediate successors of $C$ in $\mct'$. Since $C$ does not witness (1), there is a countable subset $\mct''$ of $\{C'\in \mct: C'\subset C\}$ such that $|A\cap C\setminus \bigcup \mct''|\leq \omega$. Then $succ_{\mct'}(C)$ is contained in elements and predecessors of elements in $\mct''$ at level $ht(C)+1$.
   
   Moreover, for every $C\in \mct'$, 
   \begin{enumerate}
   \item[(2)] $|A\cap C\setminus \bigcup succ_{\mct'}(C)|\leq \omega$.
   \end{enumerate}
   Note that $\mct'$ is countable since it is a countably branching tree  of height $\leq \omega$. So
   $$\bigcup\{A\cap C\setminus \bigcup succ_{\mct'}(C): C\in \mct'\}\text{  is countable}.$$   
   Now fix $\alpha\in A\setminus \bigcup\{A\cap C\setminus \bigcup succ_{\mct'}(C): C\in \mct'\}$.
   Then for every $C\in \mct'$, 
   \begin{enumerate}
   \item[(3)] if $\alpha\in C$, then $\alpha\in C'$ for some $C'\in succ_{\mct'}(C)$.
   \end{enumerate}
    So we inductively choose a chain 
    $$\omega_1=C_0\supset C_1\supset\cdot\cdot\cdot\text{ in $\mct'$ with  }\alpha\in \bigcap_{n<\omega} C_n.$$ But this contradicts Lemma \ref{lem4}. This contradiction shows the existence of some $C\in \mct$ satisfying (1) and finishes the proof of the lemma. 
  \end{proof}
  Above proof indicates that we need the full strength of Lemma \ref{lem4}. In other words, $\bigcap_{n<\omega} C_n$ being small (e.g., countable) for every chain $C_0\supset C_1\supset\cdot\cdot\cdot$ may be sufficient to control the height of $\mct$, but not strong enough for our purpose.
  \begin{lem}\label{lem10}
    Assume $\varphi_0(\pi, \mc{C}, \mc{R}, \mc{T}, \mathbf{E})$. Suppose $\msa\subset [\omega_1]^{N_\msa}$ is uncountable non-overlapping and
    \begin{itemize}
    \item $\{\msa_i: i<N_\msa\}$ is $\bfe$-invariant;
    \item  for every $i<N_\msa$ and every $A\subset \Phit(\msa_i)$ in $\mct$, $|\msa_i\cap A| \leq \omega$.
    \end{itemize}
     Then there is $\msb\in [\msa]^{\omega_1}$ that is a candidate for $\mc{C}$. 
      \end{lem}
  \begin{proof}
  First denote
    \begin{enumerate}
  \item  $\msa=\{a_\alpha: \alpha<\omega_1\}$ such that $a_\alpha< a_\beta$ whenever $\alpha<\beta$.
  \end{enumerate}
  
  Then find $\Gamma\in [\omega_1]^{\omega_1}$ such that
    \begin{enumerate}
    \item[(2)]  for every $\alpha<\omega_1$, $\alpha<\Gamma(\alpha)$;
  \item[(3)] for every $i, j<N_\msa$, and every $A\supseteq \Phit(\msa_j)$ in $\mct$,  either $\msa_i[\Gamma]\subset A$ or $\msa_i[\Gamma]\cap A=\emptyset$.
   \end{enumerate}
   Note that there are only finitely many $i, j, A$'s in (3) and (2), (3) are preserved by going to uncountable subset of $\Gamma$. So the existence of $\Gamma$ follows from a simple finite induction.
   
   Now take $\msb=\{a_\alpha: \alpha\in \Gamma\}$. By Lemma \ref{lem7} (2), for every $i<N_\msa$, 
   $$\Phit(\msb_i)=\Phit(\msa_i).$$ Consequently, Definition \ref{candidate} (ii) for $\msb$ holds.
   
   By Lemma \ref{lem6} and the fact that $\{\msa_i: i<N_\msa\}$ is $\bfe$-invariant, $\{\msb_i: i<N_\msa\}$ is $\bfe$-invariant. Definition \ref{candidate} (iii) for $\msb$ follows from (2).  Definition \ref{candidate} (iv) follows from (3) and the choice of $\msb$.
Now $\msb$ is as desired.
     \end{proof}
  
  Now we are ready to introduce the property that is preserved in the iteration process.
  \begin{defn}\label{varphi}
  Let $\varphi(\pi, \mc{C}, \mc{R}, \mc{T}, \mathbf{E})$ be the assertion that $\varphi_0(\pi, \mc{C}, \mc{R}, \mc{T}, \mathbf{E})$ together with the following statement hold.
  \begin{itemize}
  \item[(Res)] Suppose $\msa$ is a candidate for $\mc{C}$ and  $f$ is a function from $N_\msa\times N_\msa$ to 2. Then $f$ can be realized in $\msa$ if the following requirement is satisfied:
  \begin{itemize}
  \item[(i)] for every $\msb\in \mc{C}$ and  every $I, J$ in $[N_\msa]^{N_\msb}$  with $\{a[I], a[J]\}\subset \msb$ whenever $a\in \msa$,
  $f\up_{I\times J}$ is order isomorphic to some $h$ in $\mc{R}(\msb)$.
  \end{itemize}
  \end{itemize}
  Say $f$ is \emph{$\mc{R}$-satisfiable for $\msa$} if above condition (i) is satisfied.
  \end{defn}
   \textbf{Remark.} For $I\subseteq N_\msa$, there are at most finitely many $\msb$'s in $\mc{C}$ such that $a[I]\in\msb$ whenever $a\in \msa$ and they form a $\supseteq$-chain by (T3) and (C2). So only finitely many $\msb$'s are relevant for above condition (i).
  
    \textbf{Remark.} In above definition, $I$ and $J$ have one of the following relations: $I=J$, $\max I<\min J$, $\max J< \min I$. And if $f$ is $\mc{R}$-satisfiable for $\msa$, then $f$ is $\mc{R}$-satisfiable for $\msa'$ for every $\msa'\in[\msa]^{\omega_1}$.\medskip

  Recall that $\mc{R}(\msb)$ collects patterns that elements in $\msb$ follow. So it  is necessary to require $f$ to be $\mc{R}$-satisfiable for $\msa$. And roughly speaking, condition (Res) simply says that every assignment not contradicting requirements described by $\mc{R}$ can be realized.
  
  If our first ccc poset is $\pah$ as in Lemma \ref{lem1}, $\mc{C}^1=\{\bigcup G\}$, $\mc{R}^1(\bigcup G)=H$ where $G$ is $\pah$-generic and $\mct^1$, $\bfe^1$ are induced from $\mc{C}^1$, then Lemma \ref{lem1} states that (Res) in $\varphi(\pi, \mc{C}^1, \mc{R}^1, \mct^1, \bfe^1)$ holds.
  
  For some  $\mc{C}$-candidate  $\msa$ and some collection $H$ of $\mc{R}$-satisfiable for $\msa$ functions, forcing with $\pah$ will destroy $\varphi(\pi, \mc{C}, \mc{R}, \mct, \bfe)$ in general. But for some expanded $(\mc{C}', \mc{R}', \mct', \bfe')$, $\varphi(\pi, \mc{C}', \mc{R}', \mct', \bfe')$ remains true.
  First note   the following fact.
  \begin{lem}\label{lem11}
  Assume $\varphi(\pi, \mc{C}, \mc{R}, \mct, \bfe)$. Suppose $\msa$ is a candidate for $\mc{C}$ and $H$ is a non-empty collection of $\mc{R}$-satisfiable for $\msa$ functions. Then $\pah$ is ccc.
  \end{lem}
  \begin{proof}
  Fix $\{p_\alpha\in \pah: \alpha<\omega_1\}$. Going to an uncountable subset, assume $p_\alpha$'s form a $\Delta$-system. Since the root does not affect compatibility, we may assume that the root is $\emptyset$. 
  
  Choose $n$ and $\Gamma\in [\omega_1]^{\omega_1}$ such that 
  \begin{enumerate}
  \item $|p_\alpha|=n$ for every $\alpha\in \Gamma$;
  \item $\msb=\{b_\alpha: \alpha\in \Gamma\}$ is non-overlapping where $b_\alpha=\bigcup p_\alpha$.
  \end{enumerate}
  By Definition \ref{candidate} (ii) for $\msa$ and Lemma \ref{lem7} (2), 
  \begin{enumerate}
  \item[(3)] for $i<n$ and $j<N_\msa$, $\Phit(\msb_{iN_\msa+j})=\Phit(\msa_j)$.
  \end{enumerate}
  Together with Definition \ref{candidate} (i) for $\msa$, we conclude that $\{\msb_j: j<nN_\msa\}$ is $\bfe$-invariant. Also Definition \ref{candidate} (ii)-(iv) for $\msb$ follow from (3) and the fact that $\msb_{iN_\msa+j}\subseteq \msa_j$ whenever $i<n$ and $j<N_\msa$.
  
  So $\msb$ is a $\mc{C}$-candidate. Fix $h: nN_\msa\times nN_\msa \ra 2$ such that for every $i, j<n$,
  $$h\up_{[iN_\msa, (i+1)N_\msa)\times [jN_\msa, (j+1)N_\msa)} \text{ is order isomorphic to some } g\in H.$$
  
  \textbf{Claim.}   $h$ is $\mc{R}$-satisfiable for $\msb$. 
  \begin{proof}[Proof of Claim.]
  Fix $\ms{C}\in \mc{C}$ and $I, J$ in $[N_\msb]^{N_\ms{C}}$ such that $\{b[I], b[J]\}\subset \ms{C}$ whenever $b\in \msb$.   It suffices to prove that $h\up_{I\times J}$ is order isomorphic to some $f\in \mc{R}(\ms{C})$.
  
  Clearly, elements in $\{\msb_i: i\in I\}$ are pairwise $\bfe$-equivalent. Note also that  for every $i<n$,  $\{\msb_j: j\in [iN_\msa, (i+1)N_\msa)\}$ is $\bfe$-invariant. It follows that
   \begin{enumerate}
  \item[(4)] $I\subseteq [iN_\msa, (i+1)N_\msa)$ for some $i<n$ and similarly $J\subseteq [jN_\msa, (j+1)N_\msa)$ for some $j<n$.
  \end{enumerate}
 By Lemma \ref{lem8} (4), $\{a[\{k-iN_\msa: k\in I\}], a[\{k-jN_\msa: k\in J\}]\}\subset \ms{C}$ whenever $a\in \msa$.  
  Let $g\in H$ be order isomorphic to $h\up_{[iN_\msa, (i+1)N_\msa)\times [jN_\msa, (j+1)N_\msa)}$. Since $g$ is $\mc{R}$-satisfiable for $\msa$, $g\up_{\{k-iN_\msa: k\in I\}\times \{k-jN_\msa: k\in J\}}$ is order isomorphic to some $f\in \mc{R}(\ms{C})$. So $h\up_{I\times J}$ is order isomorphic to the same $f$.
  \end{proof}
   Now apply (Res) to get $\alpha<\beta$ in $\Gamma$ such that $h$ is order isomorphic to some $\pi\up_{b_\alpha\times b_\beta}$. By our choice of $h$, $p_\alpha$ is compatible with $p_\beta$ and hence  $\pah$ is ccc.
  \end{proof}
  
  Now we are ready to show the preservation of $\varphi$ by expanding $(\mc{C}, \mc{R}, \mct, \bfe)$.
\begin{lem}\label{lem pah}
Assume $\varphi(\pi, \mc{C}, \mc{R}, \mct, \bfe)$ and $|\mct|\leq \omega_1$. Suppose $\msa$ is a candidate for $\mc{C}$ and $H$ is a non-empty collection of $\mc{R}$-satisfiable for $\msa$ functions.  If $G$ is $\pah$-generic and uncountable, then  $\varphi(\pi, \mc{C}\cup\{\bigcup G\}, \mc{R}\cup\{(\bigcup G, H)\}, \mct', \bfe')$ holds in $V[G]$ where $\mct', \bfe'$ are induced from $\mc{C}\cup\{\bigcup G\}$.
\end{lem}
\begin{proof}
Denote $\mc{C}'=\mc{C}\cup\{\bigcup G\}$ and $\mc{R}'=\mc{R}\cup\{(\bigcup G, H)\}$.

Fix $k$, a sequence $\langle \msa^i\in \mc{C}: i<k\rangle$  and a partition $N_\msa=I\cup \bigcup_{i<k} I_i$ as in Lemma \ref{lem8} (3).

By Lemma \ref{lem8} (2) and the fact $\bigcup G\subseteq \msa$, $\bigcup G$ is a candidate for $\mc{C}$.

We first check $\varphi_0(\pi, \mc{C}\cup\{\bigcup G\}, \mc{R}\cup\{(\bigcup G, H)\}, \mct', \bfe')$.

(C1) is automatically satisfied. 

(C2). Fix $\ms{C}\neq \msb$ in $ \mc{C}\cup\{\bigcup G\}$ with $\msb_i\subseteq \ms{C}_j$ for some $i,j$. The only non-trivial case is $\msb=\bigcup G$ and $\msb_i\subset \ms{C}_j$. Then $\ms{C}_j\supset \msa_{i}$.  
Since $\{\msa_{i'}: i'<N_\msa\}$ is $\bfe$-invariant,
\begin{enumerate}
\item for every $j'<N_\ms{C}$, there is $i'<N_\msa$ such that $\ms{C}_{j'}\supset \msa_{i'}$.
\end{enumerate}

Together with  Definition \ref{candidate} (iv), we observe the following.
\begin{enumerate}
\item[(2)] If $b(j'')\in \ms{C}_{j'}$ for some $j'<N_\ms{C}, j''<N_\msa$ and $b\in \msa$, then $\msa_{j''}\subset \ms{C}_{j'}$  and hence  $\ms{C}_{j'}\supseteq \Phit(\msa_{j''})$ by definition of $\Phit$.
\end{enumerate}

Consequently, for $j''\in I$ and $b\in \msb$, 
$$b(j'')\notin \bigcup \ms{C}.$$

Now fix $i'<k$. Then we discuss according to the following two cases.
\begin{itemize}
\item[(i)] For some $j'<N_\ms{C}$ and $j''< |I_{i'}|$, $\ms{C}_{j'}\supseteq \Phit(\msa_{I_{i'}(j'')}=\msa^{i'}_{j''}$.
\item[(ii)] Above case does not occur.
\end{itemize}

If case (i) occurs, then by (C2) for $\mc{C}$, there are $k'$, $\{J_l\in [|I_{i'}|]^{N_\ms{C}}: l<k'\}$ such that $[a]^{N_\ms{C}}\cap \ms{C}=\{a[J_l]: l<k'\}$ and $a(j')\notin \bigcup \ms{C}$ whenever $a\in \msa^{i'}$ and $j'\in |I_{i'}|\setminus \bigcup_{l<k'} J_l$. In particular,  for every $b\in \msb$,
$$\text{for every } l<k' \text{ and }j'\in I_{i'}\setminus \bigcup_{m<k'} I_{i'}[J_m], \  b[I_{i'}[J_l]]\in \ms{C}\text{ and }b(j')\notin \bigcup \ms{C}.$$

If case (ii) occurs, 
then by (2),
$$\text{for every $b\in \msa\supset\msb$ and $j''\in I_{i'}$}, \ b(j'')\notin \bigcup \ms{C}.$$

Now it is clear that (C2) holds for $\mc{C}\cup \{\bigcup G\}$.
 
 (R1) for $\mc{R}\cup\{(\bigcup G, H)\}$ is automatically satisfied.
 
 (R2). Fix $\ms{C}\neq \msb$ in $ \mc{C}\cup\{\bigcup G\}$ with $\msb_i\subset \ms{C}_j$ for some $i,j$. The only non-trivial case is $\msb=\bigcup G$. Fix $J ,J'$ in $[N_\msb]^{N_\ms{C}}$ such that $\{b[J], b[J']\}\subset \ms{C}$ for every  $b\in \msb$ and $h\in H$. By Lemma \ref{lem8} (4) (under assumption $\varphi_0(\pi, \mc{C}, \mc{R}, \mct,\bfe)$), $\{b[J], b[J']\}\subset \ms{C}$ for every  $b\in \msa$. Since $h$ is $\mc{R}$-satisfiable for $\msa$, $h\up_{J\times J'}$ is order isomorphic to some $h'\in \mc{R}(\ms{C})$.
  
 (R3) follows from the forcing condition.
 
 (T1)-(T2) are automatically satisfied.
 
 (T3)-(T4) for $\mc{T}'$ follow from (T3)-(T4) for $\mct$ and the fact that $\bigcup G$ is a candidate for $\mc{C}$.
 
 (E1) is automatically satisfied.
 
 Then we check condition (Res). Suppose towards a contradiction that in $V[G]$, for a $\mc{C}'$-candidate    $\msb$ and a $\mc{R}'$-satisfiable for $\msb$ function $f: n\times n\ra 2$ where $n=N_\msb$, $f$ cannot be realized in $\msb$. Since we have proved $\varphi_0(\pi, \mc{C}', \mc{R}', \mct', \bfe')$, by Lemma \ref{lem8}, find $m<\omega$, a sequence $\langle \msb^i\in \mc{C}': i<m\rangle$ and a partition $n=J\cup \bigcup_{i<m} J_i$ such that
 \begin{enumerate}
    \item[(3)] $\Phi_{\mct'}(\msb_j)=\omega_1$ for $j\in J$;
    \item[(4)]  $\Phi_{\mct'}(\msb_{J_i(j)})=\msb^i_j$ for $i<m$ and  $j<|J_i|=N_{\msb^i}$;
    \item[(5)]  $a[J_i]\in \msb^i$ for $a\in \msb$ and $i<m$.
    \end{enumerate}

 We now work in $V$. First  let $p\in G$ force these properties for $\pah$-names $\dot{\msb}$ and $\dot{\msb}^i$'s such that each $\dot{\msb}^i$ is either a canonical name of some $\msb^i\in \mc{C}$ or the name of $\bigcup G$ and
 \begin{enumerate}
 \item[(6)] $p\Vdash $  $f$ can not be realized in $\dot{\msb}$.
 \end{enumerate}
 For every $\alpha<\omega_1$, find $p_\alpha\leq p$ and $b_\alpha\in [\omega_1\setminus \alpha]^n$ such that
 \begin{enumerate}
 \item[(7)] $p_\alpha\Vdash b_\alpha\in \dot{\msb}$. Extending $p_\alpha$ if necessary, assume $b_\alpha[J_i]\in p_\alpha$ if $\dot{\msb}^i$ is the name of $\bigcup G$.
 \end{enumerate}
 Find $\Gamma\in [\omega_1]^{\omega_1}$ such that
  \begin{enumerate}
 \item[(8)] $\{p_\alpha: \alpha\in \Gamma\}$ forms a $\Delta$-system with root $p'$;
 \item[(9)] $\ms{C}=\{c_\alpha: \alpha\in \Gamma\}$ is non-overlapping where $c_\alpha=b_\alpha\cup\bigcup (p_\alpha\setminus p')$ and the structures $(c_\alpha, b_\alpha, \bigcup (p_\alpha\setminus p'), \in)$'s are isomorphic for $\alpha\in \Gamma$.
 \end{enumerate}
 Now inductively apply Lemma \ref{lem candidate} $N_\ms{C}$ times to find $\Gamma'\in [\Gamma]^{\omega_1}$ such that for $\ms{C}'=\{c_\alpha: \alpha\in\Gamma'\}$, 
  \begin{enumerate}
 \item[(10)] for every $i<N_\ms{C}$ and every $C\subset \Phit(\ms{C}'_i)$ in $\mct$, $|\ms{C}'_i\cap C|\leq \omega$.
 \end{enumerate}
 Note that if (10) holds for $i$, then for every $C'\in [\ms{C}'_i]^{\omega_1}$, $\Phit(C')=\Phit(\ms{C}'_i)$ which in turn implies that (10) holds with $\ms{C}'_i$ replaced by $C'$. So the $i$th step in the inductive application of Lemma \ref{lem candidate} preserves (10) for every $i'<i$.
 
 Denote $\ms{D}=\{b_\alpha: \alpha\in \Gamma'\}$. \medskip
 
 \textbf{Claim 1.}  For every $i<m$ such that $\dot{\msb}^i$ is the canonical name of some $\msb^i\in \mc{C}$,
 \begin{itemize}
 \item[(i)] for every $\alpha\in \Gamma'$, $b_\alpha[J_i]\in \msb^i$;
 \item[(ii)] for every $j< |J_i|$, $\Phit(\ms{D}_{J_i(j)})=\msb^i_j$;
 \item[(iii)] $\{\{b_\alpha(j): \alpha\in \Gamma'\}: j\in J_i\}$ is $\bfe$-invariant.
 \end{itemize} 
 \begin{proof}[Proof of Claim 1.]
 Note that (i) follows from (7) and (iii) follows from (i)-(ii). To prove (ii), choose a $\pah$-generic filter $G'$ such that 
 $$\Sigma=\{\alpha\in \Gamma': p_\alpha\in G'\} \text{ is uncountable}.$$
 Then $\{b_\alpha: \alpha\in \Sigma\}\subseteq \dot{\msb}^{G'}$ where $\dot{\msb}^{G'}$ is the interpretation of $\dot{\msb}$ by $G'$.
 
Fix $j<|J_i|$. Clearly, $p\in G'$.  Then 
$$\Phi_{\dot{\mct'}^{G'}}(\dot{\msb}^{G'}_{J_i(j)})=\msb^i_j.$$ 
Since $\dot{\msb}^{G'}$ is a candidate for   $\dot{\mc{C}'}^{G'}$, by Definition \ref{candidate} (ii) and Lemma \ref{lem7} (2), 
$$\Phi_{\dot{\mct'}^{G'}}(\{b_\alpha({J_i(j)}): \alpha\in \Sigma\})=\msb^i_j.$$
Then by definition of $\Phit$ and $\Phi_{\dot{\mct'}^{G'}}$,  
$$\Phit(\{b_\alpha({J_i(j)}): \alpha\in \Sigma\})=\msb^i_j.$$
 By (10) and Lemma \ref{lem7} (2), $\Phit(\ms{D}_{J_i(j)})=\Phit(\{b_\alpha({J_i(j)}): \alpha\in \Sigma\})=\msb^i_j$.
 \end{proof}

 The same argument shows the following.
   \begin{enumerate}
 \item[(11)]  For $j\in J$, $\Phit(\ms{D}_j)=\omega_1$.  
 \end{enumerate}
 
 Now arbitrarily choose $J'\in [N_\ms{C}]^{N_\msa}$ such that $c_\alpha[J']\in p_\alpha$ for some $\alpha\in \Gamma'$.  
 Denote $\ms{D}'=\{c_\alpha[J']: \alpha\in \Gamma'\}\subseteq \msa$.
 
 Recall that $k$, $\langle \msa^i: i<k\rangle$ and $N_\msa=I\cup\bigcup_{i<k} I_i$ are guaranteed by Lemma \ref{lem8} (3) (under assumption $\varphi_0(\pi, \mc{C}, \mc{R}, \mct,\bfe)$). Consequently,
 \begin{enumerate}
 \item[(12)]  $\Phit(\ms{D}'_j)=\omega_1$ for $j\in I$;
 \item[(13)]   $\Phit(\ms{D}'_{I_i(j)})=\msa^i_j$ for $i<k$ and $j<|I_i|$, and $\{\{d(j): d\in \ms{D}'\}: j\in I_i\}$ is $\bfe$-invariant.
  \end{enumerate}
  
  Now by Claim 1 and (11)-(13), $\{\ms{C}'_i: i<N_\ms{C}\}$ is $\bfe$-invariant. Together with (10), we apply Lemma \ref{lem10} to obtain a $\mc{C}$-candidate $\ms{C}''\in [\ms{C}']^{\omega_1}$. Assume
  $$\ms{C}''=\{c_\alpha: \alpha\in \Gamma''\} \text{ for some } \Gamma''\subseteq \Gamma'.$$
  
  We are going to find a $\mc{R}$-satisfiable for $\ms{C}''$ function $g: N_\ms{C}\times N_\ms{C}\ra 2$  to realize $f$. We will need the following fact.\medskip
  
    \textbf{Claim 2.}   If $a\cap b_\xi\neq\emptyset$ for some $a\in p_\xi$ and $\xi\in \Gamma'$, then $a\subseteq b_\xi$.
 \begin{proof}[Proof of Claim 2.]
 Choose $G'$ and $\Sigma$ as in Claim 1.

 Let $K$ be such that $a=c_\xi[K]$. Note that by (9), $c_\alpha[K]\in p_\alpha$  for all $\alpha\in \Gamma'$.
Then 
$$\{c_\alpha[K]: \alpha\in \Sigma\}\subseteq \bigcup G'.$$ 
Let $i$ and $i'$ be such that $a(i)=b_\xi(i')$. Then $\dot{\msb}^{G'}_{i'}\cap (\bigcup G')_{i}$ is uncountable. By Definition \ref{candidate} (ii) for $\dot{\msb}^{G'}$, $\dot{\msb}^{G'}_{i'}\subset (\bigcup G')_{i}$.

  Note that $\{\dot{\msb}^{G'}_i: i<n\}$ is $\dot{\bfe'}^{G'}$-invariant. Since $c_\alpha[K]\in \bigcup G'$,  $b_\alpha\in \dot{\msb}^{G'}$ and $c_\alpha(K(i))=b_\alpha(i')$ for $\alpha\in \Sigma$,  $c_\alpha[K]\subseteq b_\alpha$ whenever $\alpha\in \Sigma$.  By (9), $a\subseteq b_\xi$.
 \end{proof}
 
 \textbf{Claim 3.} There exists $g: N_\ms{C}\times N_\ms{C}\ra 2$ such that 
 \begin{itemize}
 \item[(14)] $g$ is $\mc{R}$-satisfiable for $\ms{C}''$; 
 \item[(15)] $g\up_{K\times K}$ is order isomorphic to $f$ where $c_\alpha[K]=b_\alpha$ for   $\alpha\in \Gamma''$;
 \item[(16)] If $c_\alpha[K']\in p_\alpha$ and $c_\alpha[K'']\in p_\alpha$ for some $K', K''\subseteq N_\ms{C}$ and $\alpha\in \Gamma''$, then $g\up_{K'\times K''}$ is order isomorphic to some $h\in H$.
 \end{itemize}
 
 \begin{proof}[Proof of Claim 3.]
 Choose $c<c'$ in $\ms{C}''$.  Now choose $g: N_\ms{C}\times N_\ms{C}\ra 2$ to satisfy (15)-(16) above and the following requirement.
 \begin{itemize} 
 \item[(17)] For $(i, j)$ on which $g$ is not defined according to (15) or (16), define $g(i, j)=\pi(c(i), c'(j))$.
 \end{itemize}
 We first check that (15) and (16) can be satisfied simultaneously and hence $g$ is well-defined. Recall by Claim 2,   $K'$ (or $K''$) is either contained in $K$ or disjoint from $K$.

  Now if $K'\cap K=\emptyset$ or $K\cap K''=\emptyset$, then (15) and (16) for $K', K''$ can be satisfied simultaneously. So suppose $K'\subseteq K$ and $K''\subseteq K$. We need   $g\up_{K'\times K''}$ to be order isomorphic to $f\up_{L'\times L''}$ and order isomorphic to some $h\in H$  where $K[L']=K'$ and $K[L'']=K''$.  It suffices to show that $f\up_{L'\times L''}$ is order isomorphic to some $h\in H$. Recall that    
  $$p_\alpha\Vdash \{b_\alpha[L'], b_\alpha[L'']\}=\{c_\alpha[K'], c_\alpha[K'']\}\subset \bigcup \dot{G } \text{ and } b_\alpha\in \dot{\msb}.$$
 Then by Lemma \ref{lem8} (4), some $p'\leq p$ forces that $\{b[L'], b[L'']\}\subset \bigcup \dot{G}$ whenever $b\in \dot{\msb}$.
  Together with the fact that $f$ is forced to be $\dot{\mc{R}'}$-satisfiable for $\dot{\msb}$, we conclude that $f\up_{L'\times L''}$ is order isomorphic to some $h\in \dot{\mc{R}}'(\bigcup \dot{G})=H$. So $g$ is well-defined.\medskip
  
  Then we check that $g$ is $\mc{R}$-satisfiable for $\ms{C}''$. Now fix $\msb'\in \mc{C}$ and $K', K''$ in $[N_{\ms{C}''}]^{N_{\msb'}}$ with $\{c_\alpha[K'], c_\alpha[K'']\}\subset \msb'$ whenever $\alpha\in \Gamma''$. We want $g\up_{K'\times K''}$ to be order isomorphic to some $h\in \mc{R}(\msb')$.
  
  Note that for $i, j\in K'$, $\{c_\alpha(i): c_\alpha\in \ms{C}''\}\bfe \{c_\alpha(j): c_\alpha\in \ms{C}''\}$. Together with Claim 1, Claim 2 and (10)-(13), we conclude
  \begin{itemize}
  \item[1)] either $K'\subseteq K$;
  \item[2)] or for some $L'$ disjoint from $K$, $K'\subseteq L'$ and $c_\alpha[L']\in p_\alpha$.
  \end{itemize}
  The same conclusion holds for $K''$.
  
  If case 1) occurs to both $K'$ and $K''$, then by (15) and the fact that $f$ is forced to be $\dot{\mc{R}'}$-satisfiable for $\dot{\msb}$, $g\up_{K'\times K''}$ is order isomorphic to some $h\in \mc{R}(\msb')$.
  
  If case 2) occurs to both $K'$ and $K''$, then by (16) and the fact that every $h\in H$ is $\mc{R}$-satisfiable for $\msa$, $g\up_{K'\times K''}$ is order isomorphic to some $h'\in \mc{R}(\msb')$.
  
  Now, by symmetry, assume that $K'\subseteq K$ and for some $L''$ disjoint from $K$, $K''\subseteq L''$ and $c_\alpha[L'']\in p_\alpha$. We may assume that for no $L'\supseteq K'$, $c_\alpha[L']\in p_\alpha$ since otherwise we can simply apply (16) as above. This, together with the fact that $\msa$ is a $\mc{C}$-candidate, shows that for every $L'$ with $c_\alpha[L']\in p_\alpha$, $K'\cap L'=\emptyset$.  
  Then by  (17), 
  $$g\up_{K'\times K''}\text{ is order isomorphic to } \pi\up_{c[K']\times c'[K'']}.$$
  Together with (R3) and the fact that $\{c[K'], c'[K'']\}\subset \msb'$,  we conclude that $g\up_{K'\times K''}$ is order isomorphic to some $h\in \mc{R}(\msb')$.
  
  This shows that $g$ is $\mc{R}$-satisfiable for $\ms{C}''$ and finishes the proof of the claim.
 \end{proof}
  Applying condition (Res) in $\varphi(\pi, \mc{C}, \mc{R}, \mct,\bfe)$ to $\ms{C}''$ and $g$, we get $\alpha<\beta$ in $\Gamma''$ such that $\pi\up_{c_\alpha\times c_\beta}$ is order isomorphic to $g$.
  
  By (16), $p_\alpha\cup p_\beta$ is a condition and stronger than both $p_\alpha$ and $p_\beta$. Just note that the root $p'$ does not affect compatibility.  By (15), $\pi\up_{b_\alpha\times b_\beta}$ is order isomorphic to $f$. So
  $$p_\alpha\cup p_\beta\Vdash f \text{ is realized in } \dot{\msb} \text{ witnessed by } b_\alpha \text{ and } b_\beta.$$
  But this contradicts (6). This finishes the proof of the lemma.
\end{proof}  
  \bigskip
  
  In the rest of the section, we describe the finite support iteration of ccc posets $\langle\mc{P}_\alpha, \dot{\mc{Q}}_\beta: \alpha\leq\omega_2, \beta<\omega_2\rangle$ such that
  \begin{enumerate}
  \item $\varphi(\pi, \mc{C}^\alpha, \mc{R}^\alpha, \mct^\alpha, \bfe^\alpha)$ holds in $V^{\mc{P}_\alpha}$;
  \item for $\alpha<\beta$, $\mc{C}^\alpha\subseteq \mc{C}^\beta$ and $\mc{R}^\beta$ extends $\mc{R}^\alpha$ as functions;
  \item if $\alpha$ is a limit ordinal, then $\mc{C}^\alpha=\bigcup_{\beta<\alpha} \mc{C}^\beta$ and $\mc{R}^\alpha=\bigcup_{\beta<\alpha} \mc{R}^\beta$.
  \end{enumerate} 
  
We first prove that $\varphi$ is preserved at limit stages.
\begin{lem}\label{limit stage}
Suppose $\langle\mc{P}_\alpha, \dot{\mc{Q}}_\beta: \alpha\leq\nu, \beta<\nu\rangle$ is a finite support iteration of ccc posets such that 
\begin{itemize}
\item[(i)] for $\alpha<\nu$, $|\mct^\alpha|\leq \omega_1$ and $\varphi(\pi, \mc{C}^\alpha, \mc{R}^\alpha, \mct^\alpha, \bfe^\alpha)$ holds in $V^{\mc{P}_\alpha}$;
\item[(ii)] for $\alpha<\beta$, $\mc{C}^\alpha\subseteq \mc{C}^\beta$ and $\mc{R}^\beta$ extends $\mc{R}^\alpha$ as functions;
  \item[(iii)] $\nu$ is a limit ordinal,  $\mc{C}^\nu=\bigcup_{\beta<\nu} \mc{C}^\beta$, $\mc{R}^\nu=\bigcup_{\beta<\nu} \mc{R}^\beta$.
\end{itemize} 
Then $\varphi(\pi, \mc{C}^ \nu, \mc{R}^ \nu, \mct^ \nu, \bfe^ \nu)$ holds in $V^{\mc{P}_ \nu}$.
\end{lem}
\begin{proof}
We first prove $\varphi_0(\pi, \mc{C}^ \nu, \mc{R}^ \nu, \mct^ \nu, \bfe^ \nu)$. For (T1), it suffices to prove that every $A\in \mct^\nu$ has finitely many predecessors.  But if $A$ has infinitely many predecessors, then 
\begin{itemize}
\item either there is a $\subset$-chain $A_0\subset A_1\subset\cdot\cdot\cdot$ in $\mct^\nu$ which induces $A_0(0)>A_1(0)>\cdot\cdot\cdot$ by (T4) for $\mct^\alpha$'s;
\item or there is a $\supset$-chain $A_0\supset A_1\supset\cdot\cdot\cdot$ in $\mct^\nu$ with non-empty intersection which yields a contradiction by proof of Lemma \ref{lem4}.
\end{itemize}
Since neither of above case can occur, (T1) holds.
 The rest requirements only refer to   finite collection of elements in $\mc{C}^\nu$ and $\mct^\nu$. So $\varphi_0(\pi, \mc{C}^ \nu, \mc{R}^ \nu, \mct^ \nu, \bfe^ \nu)$ follows directly from (i)-(iii).

Now we check (Res) for $\varphi(\pi, \mc{C}^ \nu, \mc{R}^ \nu, \mct^ \nu, \bfe^ \nu)$. As in Lemma \ref{lem pah}, fix $p\in \mc{P}_\nu$, a $\mc{P}_\nu$-name $\dot{\msa}$ of a candidate for $\mc{C}^\nu$  family, $n<\omega$ and a function $f: n\times n\ra 2$ such that
\begin{enumerate}
\item $p\Vdash N_{\dot{\msa}}=n, \ f $ is $\dot{\mc{R}^\nu}$-satisfiable for $\dot{\msa}$ which can not be realized in $\dot{\msa}$.
\end{enumerate}
By Lemma \ref{lem8} and extending $p$ if necessary, find $k<\omega$, a sequence $\langle \dot{\msa^i}\in \dot{\mc{C}^\nu}: i<k\rangle$ and a partition $n=I\cup \bigcup_{i<k} I_i$ such that
\begin{enumerate}\setcounter{enumi}{1}
\item $p\Vdash \Phi_{\dot{\mct^\nu}}(\dot{\msa}_j) =\omega_1$ for $j\in I$;
\item $p\Vdash \Phi_{\dot{\mct^\nu}}(\dot{\msa}_{I_i(j)}) =\dot{\msa}^i_j$ for $i<k$ and $j< |I_i|=N_{\dot{\msa^i}}$;
\item $p\Vdash a[I_i]\in \dot{\msa^i}$ for $a\in \dot{\msa}$ and $i<k$.
\end{enumerate}

Find $\delta<\nu$, extending $p$ if necessary, such that
 \begin{enumerate}
 \item[(5)] $p \Vdash \dot{\msa^i}\in \dot{\mc{C}^\delta}$ for every $i<k$. 
 \end{enumerate}

For every $\alpha<\omega_1$, find $p_\alpha\leq p$ and $a_\alpha\in [\omega_1\setminus \alpha]^n$ such that
 \begin{enumerate}
 \item[(6)] $p_\alpha\Vdash a_\alpha\in \dot{\msa}$. 
 \end{enumerate}
 Find $\Gamma\in [\omega_1]^{\omega_1}$ such that
  \begin{enumerate}
 \item[(7)] $\{supp(p_\alpha): \alpha\in \Gamma\}$ forms a $\Delta$-system with root $r$ where $supp(p_\alpha)\in [\nu]^{<\omega}$ is the support of $p_\alpha$;
 \item[(8)] $\ms{B}=\{a_\alpha: \alpha\in \Gamma\}$ is non-overlapping.
 \end{enumerate} 
Enlarging $\delta$, we may assume that
 \begin{enumerate}\setcounter{enumi}{8}
 \item $r\cup supp(p) \subset \delta$.
 \end{enumerate} 
 Choose a $\mc{P}_\delta$-generic filter $G$ such that $p\up_\delta\in G$ and
 \begin{enumerate}\setcounter{enumi}{9}
 \item $\Sigma=\{\alpha\in \Gamma: p_\alpha\up_\delta\in G\}$ is uncountable.
 \end{enumerate}
 
 Now work in $V[G]$.  First find $\langle \msa^i\in \mc{C}^\delta: i<k\rangle $ such that
  \begin{enumerate}\setcounter{enumi}{10}
 \item $\dot{\msa^i}^{G}=\msa^i$. Here by (5), we view $\dot{\msa^i}$ as a $\mc{P}_\delta$-name and $\msa^i$ as its interpretation by $G$.
 \end{enumerate} 
 Inductively apply Lemma \ref{lem candidate} $n$ times to find $\Sigma'\in [\Sigma]^{\omega_1}$ such that for $\ms{B}'=\{a_\alpha: \alpha\in\Sigma'\}$, 
  \begin{enumerate}
 \item[(12)] for every $i<n$ and every $B\subset \Phi_{\mct^\delta}(\ms{B}'_i)$ in $\mct^\delta$, $|\ms{B}'_i\cap B|\leq \omega$.
 \end{enumerate}\medskip

 \textbf{Claim 1.}  For $j\in I$, $\Phi_{\mct^\delta}(\ms{B}'_j)=\omega_1$.  
 \begin{proof}[Proof of Claim 1.]
 Choose a $\mc{P}_\nu$-generic filter $G'$ such that $G=G'\up_\delta$ and
 $$\Sigma''=\{\alpha\in \Sigma': p_\alpha\in G'\} \text{ is uncountable}.$$
 Then $\{a_\alpha: \alpha\in \Sigma''\}\subseteq \dot{\msa}^{G'}$.
 
Clearly, $p\in G'$. So 
$$\Phi_{\mct^\nu}(\dot{\msa}^{G'}_j)=\omega_1.$$  
By Definition \ref{candidate} (ii) and Lemma \ref{lem7} (2), 
$$\Phi_{\mct^\nu}(\{a_\alpha(j): \alpha\in \Sigma''\})=\omega_1.$$
 Then by definition of $\Phi_{\mct^\nu}$ and $\Phi_{\mct^\delta}$, 
 $$\Phi_{\mct^\delta}(\{a_\alpha(j): \alpha\in \Sigma''\})=\omega_1.$$
  By (12) and Lemma \ref{lem7} (2), $\Phi_{\mct^\delta}(\ms{B}'_j)=\Phi_{\mct^\delta}(\{a_\alpha(j): \alpha\in \Sigma''\})=\omega_1$.
 \end{proof}

 The same argument shows the following.
   \begin{enumerate}
 \item[(13)] For every $i<k$ and every $j< |I_i|$, $\Phi_{\mct^\delta}(\ms{B}'_{I_i(j)})=\msa^i_j$. In particular, $\{\{a_\alpha(j): \alpha\in \Sigma'\}: j\in I_i\}$ is $\bfe^\delta$-invariant.
 \end{enumerate}
 
  Now by Claim 1 and (13), $\{\ms{B}'_i: i<n\}$ is $\bfe$-invariant. Together with (12), we apply Lemma \ref{lem10} to obtain $\ms{B}''\in [\ms{B}']^{\omega_1}$ such that
  \begin{enumerate}
  \item[(14)] $\ms{B}''$ is a candidate for $\mc{C}^\delta$.
  \end{enumerate}
   Assume
  $$\ms{B}''=\{a_\alpha: \alpha\in \Sigma''\} \text{ for some } \Sigma''\subseteq \Sigma'.$$
  
      \textbf{Claim 2.}   $f$ is $\mc{R}^\delta$-satisfiable for $\msb''$.
 \begin{proof}[Proof of Claim 2.]
 Arbitrarily choose $\ms{C}\in \mc{C}^\delta$ and $J, J'$ in $[n]^{N_\ms{C}}$ such that $\{a[J], a[J']\}\subset \ms{C}$ for all $a\in \msb''$. Choose a $\mc{P}_\nu$-generic filter $G'$ such that $G=G'\up_\delta$ and
 $$\Sigma'''=\{\alpha\in \Sigma'': p_\alpha\in G'\} \text{ is uncountable}.$$
 Then $\{a_\alpha: \alpha\in \Sigma'''\}\subseteq \dot{\msa}^{G'}$. Note also $p\in G'$.
 
 So for uncountably many $a\in \dot{\msa}^{G'}$, $\{a[J], a[J']\}\subset \ms{C}$. Since $\dot{\msa}^{G'}$ is a candidate for $\mc{C}^\nu\supseteq \mc{C}^\delta$,  by Lemma \ref{lem8} (4),
 $$\text{for all } a\in \dot{\msa}^{G'}, \ \{a[J], a[J']\}\subset \ms{C}.$$
 Since $f$ is $\mc{R}^\nu$-satisfiable for $\dot{\msa}^{G'}$, $f\up_{J\times J'}$ is order isomorphic to some $h$ in  $\mc{R}^\nu(\ms{C})=\mc{R}^\delta(\ms{C})$. So $f$ is $\mc{R}^\delta$-satisfiable for $\msb''$.
 \end{proof}

  Applying condition (Res) in $\varphi(\pi, \mc{C}^\delta, \mc{R}^\delta, \mct^\delta,\bfe^\delta)$ to $\ms{B}''$ and $f$, we get $\alpha<\beta$ in $\Sigma''$ such that $\pi\up_{a_\alpha\times a_\beta}$ is order isomorphic to $f$.
  
  By (7) and (10), $p_\alpha$ is compatible with $p_\beta$.  Say $q\leq p_\alpha, p_\beta$. Then
  $$q\Vdash f \text{ is realized in } \dot{\msa} \text{ witnessed by } a_\alpha \text{ and } a_\beta.$$
 But this contradicts (1). This finishes the proof of the lemma.
\end{proof}

  We now describe the procedure at successor stage.
  
  In $V^{\mc{P}_\xi}$, if $\varphi(\pi, \mc{C}^ \xi, \mc{R}^ \xi, \mct^ \xi, \bfe^ \xi)$ holds and the ccc poset $\mc{Q}$ that we deal with at step $\xi$ preserves $\varphi(\pi, \mc{C}^ \xi, \mc{R}^ \xi, \mct^ \xi, \bfe^ \xi)$, then we simply force with $\mc{Q}$ and let 
  $$(\mc{C}^{\xi+1}, \mc{R}^{\xi+1}, \mct^{\xi+1}, \bfe^{\xi+1})=(\mc{C}^\xi, \mc{R}^\xi, \mct^\xi, \bfe^\xi).$$
  If in $V^{\mc{P}_\xi}$,  $\varphi(\pi, \mc{C}^ \xi, \mc{R}^ \xi, \mct^ \xi, \bfe^ \xi)$ holds and the corresponding ccc poset $\mc{Q}$ forces the failure of $\varphi(\pi, \mc{C}^ \xi, \mc{R}^ \xi, \mct^ \xi, \bfe^ \xi)$, then instead of forcing with $\mc{Q}$, we will force with some $\pah$ to destroy powerfully ccc of $\mc{Q}$. Then  by Lemma \ref{lem pah}, $\varphi(\pi, \mc{C}^{\xi+1}, \mc{R}^ {\xi+1}, \mct^ {\xi+1}, \bfe^ {\xi+1})$ holds in the forcing extension for some expanded $\mc{C}^{\xi+1}, \mc{R}^ {\xi+1}$ and induced $\mct^ {\xi+1}, \bfe^ {\xi+1}$.
  
We use the following lemma to find corresponding $\msa$ and $H$.
  \begin{lem}\label{lem14}
  Assume $\varphi(\pi, \mc{C}, \mc{R}, \mct, \bfe)$ and $|\mct|\leq \omega_1$. Suppose $\mc{Q}$ is a ccc poset which forces the failure of $\varphi(\pi, \mc{C}, \mc{R}, \mct, \bfe)$ and $1\leq m<\omega$. Then there are a candidate for $\mc{C}$ family $\msa$ and a non-empty collection of $\mc{R}$-satisfiable functions $H$ such that $\Vdash_{\pah} \mc{Q}^m$ is not ccc.
  \end{lem}
\begin{proof} First note that $\varphi_0(\pi, \mc{C} , \mc{R} , \mct , \bfe )$ is absolute between models with the same $\omega_1$. So 
  \begin{itemize}
  \item[(1)] $\Vdash_\mc{Q}$ condition (Res) fails.
  \end{itemize}
  Then find $q\in \mc{Q}$, $\mc{Q}$-name $\dot{\msa}$ of a candidate for $\mc{C}$ family, $n$ and $f:n\times n\ra 2$ such that 
  \begin{itemize}
  \item[(2)] $q\Vdash_\mc{Q}$  $\dot{\msa}, f$ witness the failure of condition (Res) in $\varphi(\pi, \mc{C} , \mc{R} , \mct , \bfe )$ and $n=N_{\dot{\msa}}$.
  \end{itemize}
  Find $k<\omega$, a sequence $\langle \msa^i\in \mc{C}: i<k\rangle$ and a partition $n=I\cup \bigcup_{i<k} I_i$ such that, extending $q$ if necessary, $q$ forces the following:
   \begin{itemize}
  \item[(3)]  for $j\in I$, $\Phi_{\mct }(\dot{\msa}_j)=\omega_1$;
  \item[(4)] for $i<k$ and $j< |I_i|$, $\Phi_{\mct }(\dot{\msa}_{I_i(j)})=\msa^i_j$;
  \item[(5)] for $a\in \dot{\msa}$ and $i<k$, $a[I_i]\in \msa^i$.
    \end{itemize}
  For every $\alpha<\omega_1$, choose $q'_\alpha\leq q$ and $a'_\alpha$ such that
 \begin{itemize}
  \item[(6)] $q'_\alpha\Vdash_\mc{Q} a'_\alpha\in \dot{\msa} \cap [\omega_1\setminus \alpha]^n$.
  \end{itemize}
  
Then we choose $m$ pairs.\footnote{The choice of $m$ is not important in this section. But the freedom of choosing $m$ and $a_{\alpha, i}$'s will be used in next section.}
   \begin{itemize}
  \item[(7)] Choose, for every $\alpha<\omega_1$ and $i<m$, $(q_{\alpha, i}, a_{\alpha,i})$ to be some $(q'_\xi, a'_\xi)$ such that for $\beta<\alpha$,   $a_{\beta, m-1}<a_{\alpha,0}<a_{\alpha,1}<\cdot\cdot\cdot<a_{\alpha, m-1}$.
  \end{itemize}

  For every $\alpha$, denote
  $$a_\alpha=\bigcup_{i<m} a_{\alpha, i}.$$
  Inductively applying Lemma \ref{lem candidate} $nm$ times, we get $\Gamma\in [\omega_1]^{\omega_1}$ such that for every $i<N_{\msa'}=nm$ where $\msa'=\{a_\alpha: \alpha\in \Gamma\}$,
   \begin{itemize}
  \item[(8)] for every $A\subset \Phi_{\mct }(\msa'_i)$ in $\mct $, $|\msa'_i\cap A|\leq \omega$.
  \end{itemize}\medskip

  \textbf{Claim 1.} For every $l<m$,  and $j\in I$,  $\Phi_{\mct }(\msa'_{nl+j})=\omega_1$.

  \begin{proof}[Proof of Claim 1.]
  Choose a $\mc{Q}$-generic filter $G$ such that $q\in G$ and $\Gamma'=\{\alpha\in \Gamma: q_{\alpha, l}\in G\}$ is uncountable. Then in $V[G]$,  $\{a_\alpha(nl+j): \alpha\in \Gamma'\}\subseteq \dot{\msa}^G_{j}$.  By (3), $\Phi_{\mct }(\dot{\msa}^G_j)=\omega_1.$
  Since $\dot{\msa}^G$ is a candidate for $\mc{C}$,
  $\Phi_{\mct }(\{a_\alpha(nl+j):\alpha\in \Gamma'\}=\omega_1.$ Now the claim follows from (8) and Lemma \ref{lem7} (2).
  \end{proof}
  
The same argument shows that for each $l<m$, $i<k$ and $j<|I_i|$,
  \begin{itemize}
  \item[(9)] $\Phi_{\mct }(\msa'_{nl+I_i(j)})=\msa^i_j$.
  \end{itemize}
  So Claim 1, (9) and (5) show that $\{\msa'_i: i<nm\}$ is $\bfe $-invariant. Together with (8), we apply Lemma \ref{lem10} to obtain $\msa''\in [\msa']^{\omega_1}$ that is a candidate for $\mc{C}$. Assume for some $\Gamma'\in [\Gamma]^{\omega_1}$,
  $$\msa''=\{a_\alpha: \alpha\in \Gamma'\}.$$
  
  For $l<m$, denote $\msb^l=\{a_{\alpha, l}: \alpha\in \Gamma'\}$. We will need the following fact.\medskip
  
  \textbf{Claim 2.} For every $l<m$, $f$ is  $\mc{R} $-satisfiable for $\msb^l$.

 \begin{proof}[Proof of Claim 2.]
 Fix $l<m$. Arbitrarily choose $\ms{C}\in \mc{C} $ and $J, J'$ in $[n]^{N_\ms{C}}$ such that $\{b[J], b[J']\}\subset \ms{C}$ for all $b\in \msb^l$. Choose a $\mc{Q}$-generic filter $G$ such that 
 $$\Gamma''=\{\alpha\in\Gamma': q_{\alpha, l}\in G\} \text{ is uncountable}.$$
 Then $\{a_{\alpha, l}: \alpha\in \Gamma''\}\subseteq \dot{\msa}^{G}$ and $p\in G$.
 
 So for uncountably many $a\in \dot{\msa}^{G}$, $\{a[J], a[J']\}\subset \ms{C}$. By Lemma \ref{lem8} (4),  $ \{a[J], a[J']\}\subset \ms{C}$ for all $ a\in \dot{\msa}^{G}$.
 Since $f$ is $\mc{R} $-satisfiable for $\dot{\msa}^{G}$, $f\up_{J\times J'}$ is order isomorphic to some $h$ in  $\mc{R} (\ms{C})$. So $f$ is $\mc{R} $-satisfiable for $\msb^l$.
 \end{proof}
  
  Now let $H$ be the collection of all $h: nm\times nm\ra 2$ such that
  \begin{itemize}
  \item[(10)] for some $\alpha<\beta$ in $\Gamma'$,  $h$ is order isomorphic to $\pi\up_{a_\alpha\times a_\beta}$;
  \item[(11)] for some $l<m$, $h\up_{[nl, n(l+1))\times [nl, n(l+1))}$ is order isomorphic to $f$.
  \end{itemize}
We will need the following property of $H$.\medskip

\textbf{Claim 3.} $H$ is a non-empty collection of $\mc{R} $-satisfiable for $\msa''$ functions.

\begin{proof}[Proof of Claim 3.]
By Claim 2 and (Res) in $\varphi(\pi, \mc{C} , \mc{R} , \mct , \bfe )$, there are $\alpha<\beta$ in $\Gamma'$ such that $f$ is order isomorphic to $\pi\up_{a_{\alpha, 0}\times a_{\beta, 0}}$. Then the $h: nm\times nm\ra 2$ that is order isomorphic to $\pi\up_{a_\alpha\times a_\beta}$ is in $H$. So $H$ is non-empty.

Now we check that every $h\in H$ is $\mc{R} $-satisfiable for $\msa''$. Fix $h\in H$ and $\alpha<\beta$ in $\Gamma'$ such that $h$ is order isomorphic to $\pi\up_{a_\alpha\times a_\beta}$.  

Arbitrarily choose $\ms{C}\in \mc{C} $ and $J, J'$ in $[nm]^{N_\ms{C}}$ such that $\{a[J], a[J']\}\subset \ms{C}$ for all $a\in \msa''$. In particular, $\{a_\alpha[J], a_\beta[J']\}\subset \ms{C}$. By (R3), $\pi\up_{a_\alpha[J]\times a_\beta[J']}$ is order isomorphic to some $g\in \mc{R} (\ms{C})$. So $h\up_{J\times J'}$ is order isomorphic to the same $g\in \mc{R} (\ms{C})$. This shows that $h$ is $\mc{R} $-satisfiable for $\msa''$. 
\end{proof}

Omitting a countable subset of $\msa''$, we may assume that 
\begin{itemize}
\item[(12)] $\Vdash_{\ms{P}_{\msa'', H}} \dot{G} \text{ is uncountable}$
where $\dot{G}$ is the canonical name of a $\ms{P}_{\msa'', H}$-generic filter. 
\end{itemize}

 By Lemma \ref{lem11}, $\ms{P}_{\msa'', H}$ is ccc. Then $\{\langle q_{\alpha, i} : i<m\rangle: a_\alpha\in \bigcup G\}$ is an uncountable antichain of $\mc{Q}^m$ in $V[G]$ where $G$ is $\ms{P}_{\msa'', H}$-generic. To see this, fix $a_\alpha< a_\beta$ in $\bigcup G$. By (11), for some $l<m$, $\pi\up_{a_{\alpha, l}\times a_{\beta, l}}$ is order isomorphic to $f$. Then by (2) and (6), $q_{\alpha, l}$ is incomparable with $q_{\beta, l}$.
   Then $\msa''$ and $H$ are as desired. 
   \end{proof}

So at step $\xi$, instead of forcing with $\mc{Q}$, we will force with $\ms{P}_{\msa, H}$ guaranteed by above lemma and let
$$\mc{C}^{\xi+1}=\mc{C} \cup \{\bigcup G\}, \ \mc{R}^{\xi+1}=\mc{R} \cup \{(\bigcup G, H)\}$$
 and $\mct^{\xi+1}, \bfe^{\xi+1}$ be induced from $\mc{C}^{\xi+1}$ where $G$ is $\ms{P}_{\msa, H}$-generic.
 
Then $\ms{P}_{\msa, H}$ is ccc and in $V^{\mc{P}_{\xi+1}}$, $\varphi(\pi, \mc{C}^{\xi+1}, \mc{R}^{\xi+1}, \mct^{\xi+1}, \bfe^{\xi+1})$ holds while $\mc{Q}^m$ is not ccc.

  With the help of all previous lemmas, we are now able to describe the general procedure of iterating ccc posets with finite support that produces minimal damage to a strong coloring $\pi$.
  \begin{prop}\label{prop1}
  Suppose $2^{\omega_1}=\omega_2$ and $\pi: [\omega_1]^2\ra 2$ satisfies {\rm Pr}$_0(\omega_1, 2, \omega)$. Then there is a finite support iteration of ccc posets $\langle \mc{P}_\alpha, \dot{\mc{Q}}_\beta: \alpha\leq \omega_2, \beta<\omega_2\rangle$ and $\mc{P}_\alpha$-names $\dot{\mc{C}^\alpha}, \dot{\mc{R}^\alpha}, \dot{\mct^\alpha}, \dot{\bfe^\alpha}$ for every $\alpha\leq \omega_2$ such that
  \begin{enumerate}
  \item $\Vdash_{\mc{P}_{\omega_2}} $ {\rm MA}(powerfully ccc) + $2^\omega=\omega_2$;
  \item for every $\alpha\leq \omega_2$, $\Vdash_{\mc{P}_{\alpha}} \varphi(\pi, \dot{\mc{C}^\alpha}, \dot{\mc{R}^\alpha}, \dot{\mct^\alpha}, \dot{\bfe^\alpha})$.
  \end{enumerate}
  \end{prop}
  \begin{proof}
  $\langle \mc{P}_\alpha, \dot{\mc{Q}}_\beta: \alpha\leq \omega_2, \beta<\omega_2\rangle$ is a standard iteration of ccc posets of size $\leq \omega_1$ with the following modification   at step $\alpha$.
  \begin{itemize}
  \item[(i)] If the $\alpha$th (in some standard enumeration) poset $\dot{\mc{Q}}$ preserves $\varphi(\pi, \dot{\mc{C}^\alpha}, \dot{\mc{R}^\alpha}, \dot{\mct^\alpha}, \dot{\bfe^\alpha})$, then take $\dot{\mc{Q}}$ as $ \dot{\mc{Q}}_\alpha$ and $(\dot{\mc{C}^{\alpha+1}}, \dot{\mc{R}^{\alpha+1}}, \dot{\mct^{\alpha+1}}, \dot{\bfe^{\alpha+1}})=(\dot{\mc{C}^\alpha}, \dot{\mc{R}^\alpha}, \dot{\mct^\alpha}, \dot{\bfe^\alpha})$.
  \item[(ii)] Otherwise, choose some $1\leq m<\omega$ and $\dot{\msa}, \dot{H}$  guaranteed by Lemma \ref{lem14}. Take $\ms{P}_{\dot{\msa}, \dot H}$ as $ \dot{\mc{Q}}_\alpha$ and $\dot{\mc{C}^{\alpha+1}}=\dot{\mc{C}^{\alpha}}\cup \{\bigcup \dot{G}\}$, $\dot{\mc{R}^{\alpha+1}}=\dot{\mc{R}^{\alpha}}\cup \{(\bigcup \dot{G}, \dot{H}\}$ where $\dot{G}$ is the canonical name of a $\ms{P}_{\dot{\msa}, \dot H}$-generic filter. $\dot{\mct^{\alpha+1}}$ and $\dot{\bfe^{\alpha+1}}$ are induced from  $\dot{\mc{C}^{\alpha+1}}$.
  \end{itemize} 
  Moreover, at limit step $\alpha$,
  \begin{itemize}
  \item[(iii)] $\dot{\mc{C}^\alpha}=\bigcup_{\beta<\alpha} \dot{\mc{C}^\beta}$, $\dot{\mc{R}^\alpha}=\bigcup_{\beta<\alpha} \dot{\mc{R}^\beta}$ and $\dot{\mct^\alpha}, \dot{\bfe^\alpha}$ are induced from $\dot{\mc{C}^\alpha}$.
  \end{itemize}
  Note that at each step, at most finitely many nodes can be added to $\dot{\mct^\alpha}$. So $|\dot{\mct^\alpha}|\leq \omega_1$ for every $\alpha<\omega_2$.
  Then by Lemma \ref{lem11}-\ref{limit stage}, (2) holds. (1) follows from our choice of $ \dot{\mc{Q}}_\beta$ in case (ii) and Lemma \ref{lem14}.
  \end{proof}
 If we always choose $m=1$ in Lemma \ref{lem14}, then   {\rm MA} holds in the final model in above proposition.
If we always choose a prefixed  $m$, we get a corresponding forcing axiom.
    \begin{prop}\label{prop2}
  Suppose $2^{\omega_1}=\omega_2$, $\pi: [\omega_1]^2\ra 2$ satisfies {\rm Pr}$_0(\omega_1, 2, \omega)$ and $1\leq m<\omega$. Then there is a finite support iteration of ccc posets $\langle \mc{P}_\alpha, \dot{\mc{Q}}_\beta: \alpha\leq \omega_2, \beta<\omega_2\rangle$ and $\mc{P}_\alpha$-names $\dot{\mc{C}^\alpha}, \dot{\mc{R}^\alpha}, \dot{\mct^\alpha}, \dot{\bfe^\alpha}$ for every $\alpha\leq \omega_2$ such that
  \begin{enumerate}
  \item $\Vdash_{\mc{P}_{\omega_2}} $ {\rm MA}($\{\ms{P}: \ms{P}^m$ is ccc$\}$) + $2^\omega=\omega_2$;
  \item for every $\alpha\leq \omega_2$, $\Vdash_{\mc{P}_{\alpha}} \varphi(\pi, \dot{\mc{C}^\alpha}, \dot{\mc{R}^\alpha}, \dot{\mct^\alpha}, \dot{\bfe^\alpha})$.
  \end{enumerate}
  \end{prop}

  \section{Non-productivity of ccc posets}
  In above section, we described an iteration of ccc posets with minimal damage to a strong coloring $\pi$. In the final model, MA(powerfully ccc) always holds. However, it is also possible that the final model satisfies the full MA, e.g., taking $m=1$ in Proposition \ref{prop2}. Even if we always take large $m$, it is possible that the final model still satisfies MA.
  
  In order to guarantee the failure of MA in the final model, we will find two ccc posets whose product is not ccc and preserve the ccc property of these two posets in the iteration. In other words,  ccc is not productive and hence MA fails in the final model.
  
  Say ccc is \emph{productive} if the product of any two   ccc posets is ccc.
  
  In fact, non-productivity is necessary in the final model, since MA is equivalent to the combination of MA(powerfully ccc) and ccc being productive.
  
  In this section,  we will follow the procedures in previous section except for Lemma \ref{lem14} where we will choose $m$, $\msa$ and $H$ to satisfy additional requirements.
  
  First, we introduce   two ccc posets that will witness non-productivity of ccc.
  \begin{defn}\label{defn11}
  For $\pi: [\omega_1]^2\ra 2$ and $l<2$, 
  $$\mc{H}^\pi_l=\{p\in [\omega_1]^{<\omega}: \text{ for all } \alpha<\beta \text{ in } p, \ \pi(\alpha, \beta)=l\}$$ is the collection of finite $l$-homogeneous subsets of $\pi$ ordered by reverse inclusion. We will omit supscript  $\pi$ if it is clear from context.
  \end{defn}
  Here, for $l<2$, $\mc{H}_l$ contains all singletons, i.e., $[\omega_1]^1\subset \mc{H}_l$.
  
  We would like to preserve ccc of $\mc{H}_l$ for each $l<2$ in the iterated forcing. To do this, we will preserve the following stronger property.
  \begin{defn}
  Assume $\varphi(\pi, \mc{C}, \mc{R}, \mct, \bfe)$. Then $\psi_0(\pi, \mc{C}, \mc{R}, \mct, \bfe)$ is the assertion that  for every $l<2$,  for every $\msa$ that is a candidate for $\mc{C}$, for every $I, J\subseteq N_\msa$ such that $\{a[I] ,a[J]: a\in \msa\}\subset \mc{H}_l$, there is a $\mc{R}$-satisfiable for $\msa$ function $f$ such that $f[I\times J]=\{l\}$.
  \end{defn}
  
  We first show that under some condition, $\psi_0(\pi, \mc{C}, \mc{R}, \mct, \bfe)$ together with $\varphi(\pi, \mc{C}, \mc{R}, \mct, \bfe)$ implies ccc of $\mc{H}_l$ for $l<2$.
  \begin{lem}\label{lem17}
  Assume $\varphi(\pi, \mc{C}, \mc{R}, \mct, \bfe)$, $\psi_0(\pi, \mc{C}, \mc{R}, \mct, \bfe)$ and $|\mct|\leq \omega_1$. Then $\mc{H}_l$ is ccc for each $l<2$.
  \end{lem}
  \begin{proof}
  Fix $l<2$ and $\{p_\alpha\in \mc{H}_l: \alpha<\omega_1\}$. Find $\Gamma\in [\omega_1]^{\omega_1}$ such that 
  \begin{enumerate}
  \item $\{p_\alpha: \alpha\in \Gamma\}$ forms a $\Delta$-system. Since the root does not affect compatibility, assume the root is $\emptyset$.
  \end{enumerate}
  Going to an uncountable subset, we may assume that for some $n<\omega$, 
  \begin{enumerate}\setcounter{enumi}{1}
  \item $|p_\alpha|=n$ whenever $\alpha\in \Gamma$ and  $\{p_\alpha: \alpha\in \Gamma\}$ is non-overlapping.
  \end{enumerate}
  Inductively applying Lemma \ref{lem8}, we find $\Gamma'\in [\Gamma]^{\omega_1}$ such that
   \begin{enumerate}\setcounter{enumi}{2}
  \item for every $i<n$ and every $A\subset \Phit(\{p_\alpha(i): \alpha\in \Gamma'\})$ in $\mct$, 
  $$|A\cap \{p_\alpha(i): \alpha\in \Gamma'\}|\leq \omega.$$
  \end{enumerate}
  Note that (3) remains true if we replace $\Gamma'$ by an uncountable subset of $\Gamma'$. Now, going to an uncountable subset of $\Gamma'$ if necessary, find an uncountable non-overlapping $\msa\subset [\omega_1]^{N_\msa}$ such that 
  \begin{enumerate}\setcounter{enumi}{3}
  \item $\{\msa_i: i<N_\msa\}$ is the $\bfe$-closure of $\{\{p_\alpha(i): \alpha\in \Gamma'\}: i<n\}$;
  \item for some $I\subseteq N_\msa$, $\{a[I]: a\in \msa\}=\{p_\alpha: \alpha\in \Gamma'\}$.  
  \end{enumerate}
 Together with Lemma \ref{lem7} (3), we conclude the following.
   \begin{enumerate}\setcounter{enumi}{5}
  \item For every $i<N_\msa$ and every $A\subset \Phit(\msa_i)$ in $\mct$,   $|A\cap \msa_i|\leq \omega.$
  \end{enumerate}
  Now by Lemma \ref{lem10}, find $\msb\in [\msa]^{\omega_1}$ such that $\msb$ is a candidate for $\mc{C}$. Applying $\psi_0(\pi, \mc{C}, \mc{R}, \mct, \bfe)$ to $l, \msb, I, I$, we find a $\mc{R}$-satisfiable for $\msb$ function $f$ such that $f[I\times I]=\{l\}$.
  
  By (Res) in $\varphi(\pi, \mc{C}, \mc{R}, \mct, \bfe)$, there are $a<b$ in $\msb$ such that $\pi\up_{a\times b}$ is order isomorphic to $f$. By (5), let $\alpha<\beta$ be such that $a[I]=p_\alpha$ and $b[I]=p_\beta$. Then $\pi[p_\alpha\times p_\beta]=\{l\}$. So $p_\alpha$ is compatible with $p_\beta$.
  \end{proof}
  
  To investigate the preservation of $\psi_0$ during the iteration process, we will introduce an equivalent formulation of $\psi_0$ that is easier to analyze. But first, we describe the structure induced from a $\mc{C}$-candidate  from which a $\mc{R}$-satisfiable function can be defined easily.
  \begin{defn}
   Assume $\varphi(\pi, \mc{C}, \mc{R}, \mct, \bfe)$.  For a candidate for $\mc{C}$ family $\msa$,  let  $\mc{A}(\msa)$ be the collection of all $(\msb, I, I')\in \mc{C}\times [N_\msa]^{N_\msb}\times [N_\msa]^{N_\msb}$ such that
    \begin{itemize}
    \item[(i)] $\{a[I], a[I']\}\subset \msb\text{ whenever } a\in \msa$;
    \end{itemize}
    and $\preceq_\msa$ be the partial order defined on $\mc{A}(\msa)$ by
      \begin{itemize}
      \item[(ii)] $(\msb, I, I')\preceq_\msa (\ms{C},J, J') $ if $ I\subseteq J, I'\subseteq J' $ and $ \{a[K], a[K']\}\subset \msb$ whenever $ a\in \ms{C}$ where $J[K]=I$ and $J'[K']=I'$.
        \end{itemize}
    Let $\prec_\msa$ be the strict part of $\preceq_\msa$,
     \begin{itemize}
       \item[(iii)] $(\msb, I, I')\prec_\msa (\ms{C}, J, J') \text{ if }(\msb, I, I')\preceq_\msa (\ms{C}, J, J') \text{ and }(\msb, I, I')\neq (\ms{C}, J, J') .$
    \end{itemize}
     Let $\mc{A}_{\max}(\msa) $ be the collection of $\prec_\msa-\max $ elements of $\mc{A}(\msa)$, i.e., $(\msb, I, I')\in \mc{A}_{\max}(\msa) $ if no $(\ms{C}, J, J')$ is $\prec_\msa$ above $(\msb, I, I')$.
  \end{defn}
  It is clear that $\mc{A}(\msa)$ is finite, $\preceq_\msa$ is transitive and $(\msb, K, K')\preceq_\msa (\ms{C}, L, L') \wedge (\ms{C}, L, L')\preceq_\msa  (\msb, K, K')$ is equivalent to $(\msb, K, K')= (\ms{C}, L, L') $.
  
  The following fact will be used to construct $\mc{R}$-satisfiable functions for $\msa$.
  \begin{lem}\label{lem18}
  Assume $\varphi(\pi, \mc{C}, \mc{R}, \mct, \bfe)$. Suppose $\msa$ is a candidate for $\mc{C}$. Then for $(\msb, I, I')\neq (\ms{C}, J, J')$ in $\mc{A}_{\max}(\msa) $, $(I\times I')\cap (J\times J')=\emptyset$.
  \end{lem}
  \begin{proof}
  Suppose towards a contradiction that $I(i)=J(j)$ and $I'(i')=J'(j')$ for some $i, j, i', j'$.
  
  Then $\msb_i\cap \ms{C}_j\supseteq \{a(I(i)): a\in \msa\}$ is uncountable. By (T3), either $\msb_i\subseteq \ms{C}_j$ or $\ms{C}_j\subseteq \msb_i$. By symmetry, assume $\msb_i\subseteq \ms{C}_j$.
  
  First consider the case $\msb_i= \mc{C}_j$. By (C2), $\msb=\ms{C}$. Choose $a\in \msa$. Then $a[I]\in \msb$, $a[J]\in \msb$ and $a[I]\cap a[J]\neq\emptyset$. Since $\msb$ is non-overlapping, $I=J$. Similarly, $I'=J'$. This contradicts the assumption that $(\msb, I, I')\neq (\ms{C}, J, J')$.
  
  So $\msb_i\subset \ms{C}_j$. By (C2), take $K\in [N_\msb]^{N_\ms{C}}$ such that $i\in K$ and $b[K]\in \ms{C}$ whenever $b\in \msb$.
  
  We claim that $I[K]=J$. To see this, choose $a\in \msa$. On one hand, $a[I]\in \msb$ and hence $a[I[K]]\in \ms{C}$. On the other hand, $a[J]\in \ms{C}$. Since $I(i)\in I[K]\cap J$, $a[I[K]]\cap a[J]\neq\emptyset$. This, together with the fact that $\ms{C}$ is non-overlapping,  shows $I[K]=J$.
  
  Now consider the pair $\msb_{i'}$ and $\ms{C}_{j'}$. Since $\msb_i\subset \ms{C}_j$ and $\msb_{i'}\cap\ms{C}_{j'}$ is uncountable, $\msb_{i'}\subset\ms{C}_{j'}$. Repeating above argument we find $K'$ such that $I'[K']=J'$ and $b[K']\in \ms{C}$ whenever $b\in\msb$.
  
  This shows that $(\ms{C}, J, J')\prec_\msa (\msb,I, I')$.  But this contradicts the fact that $(\ms{C}, J, J')$ is $\prec_\msa -\max$.
  \end{proof}
  
  The following lemma describes the method of forming $\mc{R}$-satisfiable functions.
  \begin{lem}\label{lem R-sat}
  Assume $\varphi(\pi, \mc{C}, \mc{R}, \mct, \bfe)$. Suppose $\msa$ is a $\mc{C}$-candidate and $f: N_\msa\times N_\msa\ra 2$.  Then $f$ is $\mc{R}$-satisfiable for $\msa$ iff for every $(\msb, I, I')\in \mc{A}_{\max}(\msa)$, $f\up_{I\times I'}$ is order isomorphic to some $g\in \mc{R}(\msb)$.
  \end{lem}
  \begin{proof}
  Only the ``if'' part needs a proof. To prove that $f$ is $\mc{R}$-satisfiable for $\msa$, fix $\msb\in \mc{C}$ and $I, I'$ in $[N_\msa]^{N_\msb}$ such that $\{a[I], a[I']\}\subset \msb$ whenever $a\in \msa$. It suffices to show that $f\up_{I\times I'}$ is order isomorphic to some $g\in \mc{R}(\msb)$.
  
 Note that $(\msb, I, I')\in \mc{A}(\msa)$. Since $\mc{A}(\msa)$ is finite, there exists $(\ms{C}, J, J')\in \mc{A}_{\max}(\msa)$ such that $(\msb, I, I')\preceq_\msa (\ms{C}, J, J')$. Then  for some $h\in \mc{R}(\ms{C})$,
  \begin{enumerate}
  \item $ I\subseteq J, I'\subseteq J' $ and $ \{a[K], a[K']\}\subset \msb$ whenever $ a\in \ms{C}$ where $J[K]=I$ and $J'[K']=I'$;
  \item $f\up_{J\times J'}$ is order isomorphic to $h$.
  \end{enumerate}
  By (R2), $h\up_{K\times K'}$ is order isomorphic to some $g\in \mc{R}(\msb)$. Then $f\up_{I\times I'}$ is order isomorphic to the same $g$.  
  \end{proof}
  
    Now we have a new formulation of $\psi_0(\pi, \mc{C}, \mc{R}, \mct, \bfe)$ which refers to elements of $\mc{C}$ instead of candidates  for $\mc{C}$.
    \begin{lem}\label{lem psi0}
      Assume $\varphi(\pi, \mc{C}, \mc{R}, \mct, \bfe)$ and $|\mct|\leq \omega_1$. The following statements are equivalent.
      \begin{itemize}
      \item[(i)] $\psi_0(\pi, \mc{C}, \mc{R}, \mct, \bfe)$.
      \item[(ii)] For every $l<2$, every $\msa\in \mc{C}$ and every $I, J\subseteq N_\msa$ such that 
      $$|\{a\in \msa: a[I]\in \mc{H}_l\}|=|\{a\in \msa: a[J]\in\mc{H}_l\}|=\omega_1,$$
       there exists $f\in\mc{R}(\msa)$ such that $f[I\times J]=\{l\}$.
      \end{itemize}
    \end{lem}
  \begin{proof}
  ``(i)$\Rightarrow$(ii)''. Fix $l<2, \msa\in \mc{C}$ and $I, J\subseteq N_\msa$ such that 
  $$|\{a\in \msa: a[I]\in \mc{H}_l\}|=|\{a\in \msa: a[J]\in\mc{H}_l\}|=\omega_1.$$ 
  First, by Lemma \ref{lem candidate},  find $\msb\in [\msa]^{\omega_1}$ such that
   \begin{enumerate}
   \item for every $b\in \msb$, $b[I]\in \mc{H}_l$;
  \item for every $A\subset\Phit(\msb_0)$ in $\mct$, $|\msb_0\cap A|\leq \omega$.
  \end{enumerate}
  Let $\msb'\subset [\omega_1]^{N_{\msb'}}$ be uncountable non-overlapping such that
  $$\{\msb'_i: i<N_{\msb'}\} \text{ is the $\bfe$-closure of } \msb_0.$$
  By Lemma \ref{lem7},
    \begin{enumerate}\setcounter{enumi}{2}
  \item for every $i<N_{\msb'}$ and $A\subset\Phit(\msb'_i)$ in $\mct$, $|\msb'_i\cap A|\leq \omega$.
  \end{enumerate}
  Let $I'$ be such that
     \begin{enumerate}\setcounter{enumi}{3}
  \item for every $b\in \msb'$, $b[I']\in \msb$.
  \end{enumerate}
  Then by Lemma \ref{lem10}, there is $\msb''\in [\msb']^{\omega_1}$ such that $\msb''$ is a candidate for $\mc{C}$. 
  
  Similarly, find $\ms{C}\in [\msa]^{\omega_1}$,  $\ms{C}'$, $J'$ and $\ms{C}''\in [\ms{C}']^{\omega_1}$ that is a candidate for $\mc{C}$ such that
  \begin{enumerate}\setcounter{enumi}{4}
  \item for every $c\in \ms{C}$, $c[J]\in \mc{H}_l$;
  \item $\{\ms{C}'_i: i<N_{\msb'}\} \text{ is the $\bfe$-closure of } \ms{C}_0$;
   \item for every $i<N_{\ms{C}'}$ and $A\subset\Phit(\ms{C}'_i)$ in $\mct$, $|\ms{C}'_i\cap A|\leq \omega$.
  \item for every $c\in \ms{C}'$, $c[J']\in \ms{C}$.
  \end{enumerate}
  
  Now find uncountable non-overlapping $\ms{D}\subset [\omega_1]^{N_{\msb''}+N_{\ms{C}''}}$ such that
    \begin{enumerate}\setcounter{enumi}{8}
  \item for every $d\in \ms{D}$, $d[N_{\msb''}]\in \msb''$ and $d[[N_{\msb''}, N_{\msb''}+N_{\ms{C}''})]\in \ms{C}''$;
  \item  $\ms{D}$ is a candidate for $\mc{C}$.
  \end{enumerate}
  Together with (1), (4), (5) and (8), we conclude that
$$\text{for every $d\in \ms{D}$, $d[I'[I]]\in \mc{H}_l$ and }d[\{N_{\msb''}+ J'(j): j\in J\}]\in \mc{H}_l.$$

Applying $\psi_0(\pi, \mc{C}, \mc{R}, \mct, \bfe)$ to $l, \ms{D}$ and $ I'[I], \{N_{\msb''}+ J'(j): j\in J\}\subseteq N_\ms{D}$, we find a $\mc{R}$-satisfiable for $\ms{D}$ function $g$ such that $g[I'[I]\times \{N_{\msb''}+ J'(j): j\in J\}]=\{l\}$. 
  
  Now by (4) and (8), $g\up_{I'\times \{N_{\msb''}+j: j\in J'\}}$ is order isomorphic to some $f\in \mc{R}(\msa)$. Clearly, $f[I\times J]=\{l\}$.\medskip
  
  ``(ii)$\Rightarrow$(i)''.  Fix  $l<2$,  a candidate for $\mc{C}$ family $\msa$ and $I, J\subseteq N_\msa$ such that $\{a[I] ,a[J]: a\in \msa\}\subset \mc{H}_l$. Now define $f: N_\msa\times N_\msa\ra 2$ such that
   \begin{enumerate}\setcounter{enumi}{10}
  \item for every $(\msb, K, K')\in\mc{A}_{\max}(\msa)$, $f\up_{K\times K'}$ is order isomorphic to some $g\in \mc{R}(\msb)$ and $f[(K\cap I)\times (K'\cap J)]=\{l\}$.
   \item for $(i, j)$ not in $K\times K'$ of above form, $f(i, j)=l$.
  \end{enumerate}
 First note that for a single  $(\msb, K, K')\in\mc{A}_{\max}(\msa)$, there exists $f\up_{K\times K'}$ satisfying (11). To see this, apply (ii) to $l$, $\msb\in \mc{C}$ and $I' ,J'$ where $K[I']=K\cap I$ and $K'[J']=K'\cap J$.  
 
 Then by Lemma \ref{lem18}, there exists $f$ satisfying (11).  Clearly, (12) is automatically satisfied. 
 
 We will show that $f$ is as desired. 
 
 To see $f[I\times J]=\{l\}$, fix $(i, j)\in I\times J$. If $f(i, j)$ is defined by (12), then $f(i, j)=l$. Now suppose $(i, j)\in K\times K'$ for some $(\msb, K, K')\in\mc{A}_{\max}(\msa)$. Then $(i, j)\in (K\cap I)\times (K'\cap J)$ and hence $f(i, j)=l$.
 
 Now by Lemma \ref{lem R-sat}, $f$ is $\mc{R}$-satisfiable for $\msa$ and hence witnesses $\psi_0(\pi, \mc{C}, \mc{R}, \mct, \bfe)$ for $l, \msa, I, J$.
  \end{proof}
  
  The reformulation of $\psi_0(\pi, \mc{C}, \mc{R}, \mct, \bfe)$ in above lemma indicates the preservation of $\psipi$ in the iterated forcing process. We first describe the successor steps.
  \begin{cor}\label{cor preservation}
  Assume $\varphi(\pi, \mc{C}, \mc{R}, \mct, \bfe)$, $\psi_0(\pi, \mc{C}, \mc{R}, \mct, \bfe)$ and $|\mct|\leq \omega_1$. 
  \begin{enumerate}
  \item If $\ms{P}$ is a ccc poset preserving $\varphi(\pi, \mc{C}, \mc{R}, \mct, \bfe)$, then $\ms{P}$ also preserves $\psi_0(\pi, \mc{C}, \mc{R}, \mct, \bfe)$.
  \item  Suppose $\msa$ is a candidate for $\mc{C}$ and $H$ is a non-empty collection of $\mc{R}$-satisfiable functions. If for every $l<2$ and every $I, J\subseteq N_\msa$ such that
  $$|\{a\in \msa: a[I]\in \mc{H}_l\}|=|\{a\in \msa: a[J]\in\mc{H}_l\}|=\omega_1,$$
       there exists $f\in H$ such that $f[I\times J]=\{l\}$, then $\psi_0(\pi, \mc{C}\cup \{\bigcup G\}, \mc{R}\cup\{(\bigcup G, H)\}, \mct', \bfe')$ holds in $V[G]$ where $G$ is an uncountable $\pah$-generic filter and $\mct', \bfe'$ are induced from $\mc{C}\cup \{\bigcup G\}$.
  \end{enumerate}
  \end{cor}
  \textbf{Remark.} It may seem that the assumption in (2) of above corollary that corresponding $f$ exists for $l<2$ and $I, J$ with $|\{a\in \msa: a[I]\in \mc{H}_l\}|=|\{a\in \msa: a[J]\in\mc{H}_l\}|=\omega_1$ is stronger than the property asserted by $\psipi$. But if we replace $\msa$ by some $\msb\in [\msa]^{\omega_1}$ with property that for every $l<2$ and $I$, either $\{a\in \msb: a[I]\in \mc{H}_l\}=\emptyset$ or $\{a\in \msb: a[I]\in \mc{H}_l\}=\msb$, then the assumption in (2) above will be the same as the one asserted by $\psipi$. As a matter of fact, the proof of Lemma \ref{lem psi0} shows that $\psipi$ is equivalent to its modification by replacing its assumption $\{a[I] ,a[J]: a\in \msa\}\subset \mc{H}_l$ by $|\{a\in \msa: a[I]\in \mc{H}_l\}|=|\{a\in \msa: a[J]\in\mc{H}_l\}|=\omega_1$.\medskip
  
  Then we describe the limit steps. Note that the ``(ii)$\Rightarrow$(i)'' part in the proof of  of Lemma \ref{lem psi0} does not use $|\mct|\leq \omega_1$.
  \begin{cor}\label{cor limit}
Suppose $\langle\mc{P}_\alpha, \dot{\mc{Q}}_\beta: \alpha\leq\nu, \beta<\nu\rangle$ is a finite support iteration of ccc posets such that 
\begin{itemize}
\item[(i)] for $\alpha<\nu$, $|\mct^\alpha|\leq \omega_1$ and $\varphi(\pi, \mc{C}^\alpha, \mc{R}^\alpha, \mct^\alpha, \bfe^\alpha)$,  $\psi_0(\pi, \mc{C}^\alpha, \mc{R}^\alpha, \mct^\alpha, \bfe^\alpha)$ hold in $V^{\mc{P}_\alpha}$;
\item[(ii)] for $\alpha<\beta$, $\mc{C}^\alpha\subseteq \mc{C}^\beta$ and $\mc{R}^\beta$ extends $\mc{R}^\alpha$ as functions;
  \item[(iii)] $\nu$ is a limit ordinal,  $\mc{C}^\nu=\bigcup_{\beta<\nu} \mc{C}^\beta$, $\mc{R}^\nu=\bigcup_{\beta<\nu} \mc{R}^\beta$.
\end{itemize} 
Then $\varphi(\pi, \mc{C}^ \nu, \mc{R}^ \nu, \mct^ \nu, \bfe^ \nu)$ and $\psi_0(\pi, \mc{C}^ \nu, \mc{R}^ \nu, \mct^ \nu, \bfe^ \nu)$ hold in $V^{\mc{P}_ \nu}$.
\end{cor}

So in order to preserve property $\psi_0$, we only need to take care of the forcing $\pah$ in the iteration process. More precisely, in the proof of Lemma \ref{lem14}, we need to choose $m, \msa''$ and $H$ carefully so that the assumption of Corollary \ref{cor preservation} (2) is satisfied. To do this, we first observe the following property of $\mc{R}$-satisfiable functions.
\begin{lem}\label{lem23}
Assume $\varphi(\pi, \mc{C}, \mc{R}, \mct, \bfe)$. Suppose $1< m<\omega$ and $\msa$ is a candidate for $\mc{C}$. If $\langle f_{i, j}: i<j<m\rangle$ is a sequence of $\mc{R}$-satisfiable for $\msa$ functions, then there are $a_0<a_1<\cdot\cdot\cdot <a_{m-1}$ in $\msa$ such that $f_{i, j}$ is order isomorphic to $\pi\up_{a_i\times a_j}$ whenever $i<j<m$.
\end{lem}
\begin{proof}
We prove the lemma by induction on $m$. For $m=2$, the lemma follows from (Res) in $\varphi(\pi, \mc{C}, \mc{R}, \mct, \bfe)$.

Suppose the lemma holds for $m=k$ and we   prove for $m=k+1$. Fix a $\mc{C}$-candidate $\msa$ and a sequence of $\mc{R}$-satisfiable for $\msa$ functions $\langle f_{i, j}: i<j<k+1\rangle$.

First note that for every $\msa'\in [\msa]^{\omega_1}$ and $i<j<k+1$, $\msa'$ is a candidate for $\mc{C}$ and $f_{i, j}$ is  $\mc{R}$-satisfiable for $\msa'$. By induction hypothesis, we then inductively find $\{\langle a_{\alpha, i}\in \msa: i<k\rangle :\alpha<\omega_1\}$  such that for $\alpha<\beta$,
\begin{enumerate}
\item  $a_{\alpha, 0}<a_{\alpha, 1}<\cdot\cdot\cdot<a_{\alpha, k-1}<a_{\beta, 0}$;
\item $f_{i, j}$ is order isomorphic to $\pi\up_{a_{\alpha, i}\times a_{\alpha, j}}$ whenever $i<j<k$.
\end{enumerate}
At step $\alpha$, apply   induction hypothesis to $\{a\in \msa: a> a_{\xi, i} \text{ for all }\xi<\alpha, i<k\}$ and $\langle f_{i, j}: i<j<k\rangle$.

For each $\alpha$, let $a_\alpha=\bigcup_{i<k} a_{\alpha, i}$. Then
$$\msb=\{a_\alpha: \alpha<\omega_1\}$$
is a candidate for $\mc{C}$, $N_\msb=kN_\msa$ and
\begin{enumerate}\setcounter{enumi}{2}
\item for every $b\in \msb$ and $i<k$, $b[[iN_\msa, (i+1)N_\msa)]\in \msa$.
\end{enumerate}
It is   straightforward to check (or by   Lemma \ref{lem11} Claim) that  $f: N_\msb\times N_\msb\ra 2$ is $\mc{R}$-satisfiable for $\msb$ iff for every $i, j<k$, $f\up_{[iN_\msa, (i+1)N_\msa)\times [jN_\msa, (j+1)N_\msa)}$ is order isomorphic to some $g$ that is $\mc{R}$-satisfiable for $\msa$.

Now choose a $\mc{R}$-satisfiable for $\msb$ function $f$ such that
\begin{enumerate}\setcounter{enumi}{3}
\item for every $i<k$, $f\up_{[iN_\msa, (i+1)N_\msa) \times [0, N_\msa)}$ is order isomorphic to $f_{i, k}$.
\end{enumerate}
Applying (Res) in $\varphi(\pi, \mc{C}, \mc{R}, \mct, \bfe)$ to $\msb$ and $f$, we get $\alpha<\beta$ such that $\pi\up_{a_\alpha\times a_\beta}$ is order isomorphic to $f$. Then by (4),
\begin{enumerate}\setcounter{enumi}{4}
\item for every $i<k$, $\pi\up_{a_{\alpha, i}\times a_{\beta, 0}}$ is order isomorphic to $f_{i, k}$.
\end{enumerate}
This, together with (2), shows that $a_{\alpha, 0}<\cdot\cdot\cdot<a_{\alpha, k-1}<a_{\beta, 0}$ in $\msa$ are as desired. This finishes the induction and the proof of the lemma.
\end{proof}
  
  Then we will use the following lemma to choose $m, \msa''$ and $H$ such that $H$ is non-empty and the assumption of Corollary \ref{cor preservation} (2) holds.
  \begin{lem}\label{lem24}
  Assume $\varphi(\pi, \mc{C}, \mc{R}, \mct, \bfe)$ and $\psi_0(\pi, \mc{C}, \mc{R}, \mct, \bfe)$.  Suppose  $\msa$ is a candidate for $\mc{C}$. Then there are $m<\omega$ and a sequence of $\mc{R}$-satisfiable for $\msa$ functions $\langle f_{i, j}: i<j<m\rangle$ such that for all $a_0<a_1<\cdot\cdot\cdot <a_{m-1}$ in $\msa$, if  $f_{i, j}$ is order isomorphic to $\pi\up_{a_i\times a_j}$ whenever $i<j<m$, then for every $l<2$ and $b\in \mc{H}_l$, $b\cap a_i=\emptyset$ for some $i<m$.
  \end{lem}
  \begin{proof}
  Take $m=N_\msa+2N_\msa^{N_\msa}$. Let $\{s_k: k<N_\msa^{N_\msa}\}$ enumerate functions from $N_\msa$ to $N_\msa$.
  
  Then for $i<N_\msa$ and $k<N_\msa^{N_\msa}$, let $f_{i, N_\msa+k}$ be a $\mc{R}$-satisfiable for $\msa$ function such that $f_{i, N_\msa+k}(s_k(i), i)=0$ and $f_{i, N_\msa+N_\msa^{N_\msa}+k}$ be a $\mc{R}$-satisfiable for $\msa$ function such that $f_{i, N_\msa+N_\msa^{N_\msa}+k}(s_k(i), i)=1$. The existence of $f_{i, j}$'s follows from $\psi_0(\pi, \mc{C}, \mc{R}, \mct, \bfe)$. The rest $f_{i, j}$'s are arbitrary $\mc{R}$-satisfiable for $\msa$ functions.
  
  To show that $m$ and $\langle f_{i, j}: i<j<m\rangle$ are as desired, fix $a_0<a_1<\cdot\cdot\cdot <a_{m-1}$ in $\msa$ such that  $f_{i, j}$ is order isomorphic to $\pi\up_{a_i\times a_j}$ whenever $i<j<m$. Fix $l<2$ and $b\in \mc{H}_l$.
  
  If $b\cap a_i=\emptyset$ for some $i<N_\msa$, then we are done.
  
  Suppose otherwise. Define $s: N_\msa\ra N_\msa$ by 
  $$a_i(s(i))=\min (b\cap a_i) \text{ for all }i<N_\msa.$$
  Assume $s=s_k$ for some $k<N_\msa^{N_\msa}$.
  
  If $l=1$, then $b\cap a_{N_\msa+k}=\emptyset$. To see this, note that $a_{N_\msa+k}(i)\notin b$ since $\pi(a_i(s_k(i)), a_{N_\msa+k}(i))=f_{i, N_\msa+k}(s_k(i), i)=0$ and $a_i(s_k(i))=a_i(s(i))\in b$.
  
  Similarly, if $l=0$, then $b\cap a_{N_\msa+N_\msa^{N_\msa}+k}=\emptyset$. This finishes the proof of the lemma.
  \end{proof}
  
    \begin{lem}\label{lem25}
  Assume $\varphi(\pi, \mc{C}, \mc{R}, \mct, \bfe)$, $\psi_0(\pi, \mc{C}, \mc{R}, \mct, \bfe)$ and $|\mct|\leq \omega_1$. Suppose $\mc{Q}$ is a ccc poset which forces the failure of $\varphi(\pi, \mc{C}, \mc{R}, \mct, \bfe)$. Then there are $m<\omega$, a candidate for $\mc{C}$ family $\msa$ and a non-empty collection of $\mc{R}$-satisfiable functions $H$ such that in $V[G]$ where $G$ is $\pah$-generic, $\mc{Q}^m$ is not ccc and $\psi_0(\pi, \mc{C}\cup\{\bigcup G\}, \mc{R}\cup\{(\bigcup G, H)\}, \mct', \bfe')$ holds
where $\mct'$ and $\bfe'$ are induced from $\mc{C}\cup\{\bigcup G\}$.
  \end{lem}
  \begin{proof}
  As in the proof of Lemma \ref{lem14}, find $q\in \mc{Q}$, $\mc{Q}$-name $\dot{\msa}$ of a $\mc{C}$-candidate, $n$ and $f:n\times n\ra 2$ such that 
  \begin{enumerate}
  \item $q\Vdash_\mc{Q}$  $\dot{\msa}, f$ witness the failure of condition (Res) in $\varphi(\pi, \mc{C} , \mc{R} , \mct , \bfe )$ and $n=N_{\dot{\msa}}$.
  \end{enumerate}
  
  For every $\alpha<\omega_1$, choose $q'_\alpha\leq q$ and $a'_\alpha$ such that
 \begin{itemize}
  \item[(2)] $q'_\alpha\Vdash_\mc{Q} a'_\alpha\in \dot{\msa} \cap [\omega_1\setminus \alpha]^n$.
  \end{itemize}
  
As in the proof of Lemma \ref{lem14}, find $\Gamma\in [\omega_1]^{\omega_1}$ such that
\begin{enumerate}\setcounter{enumi}{2}
  \item $\msa'=\{a_\alpha': \alpha\in \Gamma\}$   is a $\mc{C}$-candidate and $f$ is $\mc{R}$-satisfiable for $\msa'$;\footnote{E.g., take $\msa''$ in the proof of Lemma \ref{lem14} with $m=1$.}
  \item for every $l<2$ and $I\subseteq n$, $\{a\in \msa': a[I]\in \mc{H}_l\}$ is either $\emptyset$ or $\msa'$.
\end{enumerate}
Now choose $m<\omega$ and $\mc{R}$-satisfiable for $\msa'$ functions $\langle f_{i, j}: i<j<m\rangle$ guaranteed by Lemma \ref{lem24}. Note that $f_{i, j}$ is $\mc{R}$-satisfiable for $\msa''$ whenever $\msa''\in [\msa']^{\omega_1}$.

Inductively find $\{\langle a_{\alpha, i}\in \msa': i<m\rangle :\alpha<\omega_1\}$  such that for $\alpha<\beta$,
\begin{enumerate}\setcounter{enumi}{4}
\item  $a_{\alpha, 0}<a_{\alpha, 1}<\cdot\cdot\cdot<a_{\alpha, m-1}<a_{\beta, 0}$;
\item for every $l<2$ and $b\in \mc{H}_l$, $b\cap a_{\alpha, i}=\emptyset$ for some $i<m$.
\end{enumerate}
To see the existence of such $a_{\alpha, i}$'s, we may   apply Lemma \ref{lem23} to $\{a\in \msa': a> a_{\xi, i}$ for all $\xi<\alpha, i<m\}$ and $\langle f_{i, j}: i<j<m\rangle$ to get $a_{\alpha, 0}<\cdot\cdot\cdot<a_{\alpha, m-1}$ and (6) follows from Lemma \ref{lem24}.

  For every $\alpha$, denote
  $$a_\alpha=\bigcup_{i<m} a_{\alpha, i}$$
  and for each $i<m$, let $q_{\alpha, i}\in \mc{Q}$ be such that $(q_{\alpha, i}, a_{\alpha, i})=(q'_\xi, a'_\xi)$ for some $\xi$.
  
  As in the proof of Lemma \ref{lem14}, find $\Gamma'\in [\omega_1]^{\omega_1}$ such that 
  \begin{enumerate}\setcounter{enumi}{6}
  \item $\msa''=\{a_\alpha: \alpha\in \Gamma'\}$   is a candidate for $\mc{C}$.
\end{enumerate}

  Now let $H$ be the collection of all $h: nm\times nm\ra 2$ such that
\begin{enumerate}\setcounter{enumi}{7}
  \item $h$ is $\mc{R} $-satisfiable for $\msa''$ functions;
  \item for some $l<m$, $h\up_{[nl, n(l+1))\times [nl, n(l+1))}$ is order isomorphic to $f$.
\end{enumerate}
Clearly, $H$ is non-empty.

We will need the following to apply Corollary \ref{cor preservation} (2).\medskip

\textbf{Claim.} For every $l<2$ and every $I, J\subseteq nm$ such that
  $$|\{a\in \msa'': a[I]\in \mc{H}_l\}|=|\{a\in \msa'': a[J]\in\mc{H}_l\}|=\omega_1,$$
       there exists $h\in H$ such that $h[I\times J]=\{l\}$.

\begin{proof}[Proof of Claim.]
By (4), $\{a\in \msa'': a[I]\in \mc{H}_l\}=\msa''=\{a\in \msa'': a[J]\in\mc{H}_l\}$.
By (6), $I\cap [ni^*, n(i^*+1))=\emptyset$ for some $i^*<m$.

Fix $(i, j)\in m\times m$ such that $[ni, n(i+1))\cap I\neq \emptyset$ and $[nj, n(j+1))\cap J\neq \emptyset$. Applying $\psipi$ to $l, \msa'$ and $I' ,J'$ where
$$I'=\{k<n: ni+k\in I\}\text{ and }J'=\{k<n: nj+k\in J\},$$ 
we get a $\mc{R}$-satisfiable for $\msa'$ function $g_{i, j}$ such that $g_{i, j}[I'\times J']=\{l\}$.

Now define $h: nm\times nm\ra 2$ by
\begin{enumerate}\setcounter{enumi}{9}
  \item for $(i, j)$ such that $[ni, n(i+1))\cap I\neq \emptyset$ and $[nj, n(j+1))\cap J\neq \emptyset$, $h\up_{[ni, n(i+1))\times [nj, n(j+1))}$ is order isomorphic to $g_{i,j}$;
  \item for the rest $(i, j)$, $h\up_{[ni, n(i+1))\times [nj, n(j+1))}$ is order isomorphic to $f$.
\end{enumerate}
By (10) and our choice of $g_{i, j}$, $h[I\times J]=\{l\}$. 

Since $h\up_{[ni, n(i+1))\times [nj, n(j+1))}$ is order isomorphic to some $\mc{R}$-satisfiable for $\msa'$ function $g$  for all $i, j<m$, it is straightforward to check that $h$ is $\mc{R}$-satisfiable for $\msa''$. 

By (11) and the fact that $I\cap [ni^*, n(i^*+1))=\emptyset$,  $h\up_{[ni^*, n(i^*+1))\times [ni^*, n(i^*+1))}$ is order isomorphic to $f$.
So $h$ is as desired.
\end{proof}

Omitting a countable subset of $\msa''$, we may assume that 
\begin{itemize}
\item[(12)] $\Vdash_{\ms{P}_{\msa'', H}} \dot{G} \text{ is uncountable}$
where $\dot{G}$ is the canonical name of a $\ms{P}_{\msa'', H}$-generic filter. 
\end{itemize}

 By Lemma \ref{lem11}, $\ms{P}_{\msa'', H}$ is ccc. Clearly, $\{\langle q_{\alpha, i} : i<m\rangle: a_\alpha\in \bigcup G\}$ is an uncountable antichain of $\mc{Q}^m$ in $V[G]$ where $G$ is $\ms{P}_{\msa'', H}$-generic.  By Corollary \ref{cor preservation} and above claim, $\psi_0(\pi, \mc{C}\cup\{\bigcup G\}, \mc{R}\cup\{(\bigcup G, H)\}, \mct', \bfe')$ holds in $V[G]$. Then $m, \msa''$ and $H$ are as desired. 
   \end{proof}
  
  \begin{proof}[Proof of Theorem \ref{theorem1}]
  Starting from a model of  $2^{\omega_1}=\omega_2$ in which there exists a coloring $\pi: [\omega_1]^2\ra 2$ witnessing {\rm Pr}$_0(\omega_1, 2, \omega)$. Repeating the proof of Proposition \ref{prop1} while replacing Lemma \ref{lem14} by Lemma \ref{lem25}, we get a model of {\rm MA}(powerfully ccc) + $2^\omega=\omega_2$ in which $\mc{H}_0$ and $\mc{H}_1$ are ccc.  In  particular, {\rm MA} fails.
  \end{proof}

  \section{Enlarging the continuum}
  In this section, we will enlarge the continuum while we distinguish {\ma}  from \map.  From the cardinal invariant point of view, $\mathfrak{m}$(powerfully ccc) can be arbitrarily large while $\mathfrak{m}=\omega_1$.
  
  For a property $\Psi$ of ccc property, $\mathfrak{m}(\Psi)$ is the least cardinal $\kappa$ such that $\mathrm{MA}_\kappa(\Psi)$ fails. $\mathfrak{m}$ is $\mathfrak{m}$(ccc). Note that $\mathfrak{m}$ has to be $\omega_1$ if $\mathfrak{m}$(powerfully ccc)$>\mathfrak{m}$.
  
  A natural attempt is to start from a model of {\rm MA}(powerfully ccc) + $\neg${\rm MA} and force with a powerfully ccc poset to enlarge the continuum while preserve \map. Then $\mc{H}_0$ and $\mc{H}_1$ remain ccc since if an uncountable antichain can be added by a powerfully ccc poset, then by \map, some uncountable antichain already exists in the ground model. But we do not know if there is always a powerfully ccc poset to enlarge the continuum while preserve \map.
  \begin{question}
  Assume $\omega_1< \kappa=\kappa^{<\kappa}$. Is there always a powerfully ccc poset $\mc{P}$ such that $V^{\mc{P}}\vDash $ {\rm MA}(powerfully ccc) + $2^\omega=\kappa$? What if we assume {\rm MA}(powerfully ccc)+ $\neg${\rm CH} in the ground model?
  \end{question}
  
  Note that we can not simply iteratively force with all powerfully ccc posets since powerfully ccc is not preserved under iteration.
  \begin{exam}\label{exam1}
  Suppose $\sigma: [\omega_1]^2\ra 2$ satisfies {\rm Pr}$_0(\omega_1, 2, \omega)$ and $Fn(\omega, 2)$ is the Cohen forcing. $\dot{Q}$ is the $Fn(\omega, 2)$-name of the poset $\mc{H}^\sigma_{r(0)}$ where $r$ is the generic Cohen real. Then $Fn(\omega, 2)$ is powerfully ccc and $\Vdash_{Fn(\omega, 2)} \dot{Q}$ is powerfully ccc. But $Fn(\omega, 2)*\dot{Q}$ is not powerfully ccc. In fact, $\{((\{(0, 0)\}, \{\alpha\}), (\{(0, 1)\}, \{\alpha\})): \alpha<\omega_1\}$ is an uncountable antichain of $(Fn(\omega, 2)*\dot{Q})^2$.
  \end{exam}
  Of course some additional assumption is needed in above example since under {\rm MA{+$\neg${\rm CH}, every finite support iteration of ccc posets is ccc and hence has precaliber $\omega_1$. Also, in Section 6.3, we will give properties together with {\rm MA}(powerfully ccc)+$\neg${\rm MA}   implying that powerfully ccc is closed under finite support iterated forcing. But the advantage of the strategy in this section is that in the iteration process, we can force with many ccc non-powerfully ccc posets.\medskip

  Looking into the proof of Proposition \ref{prop1} and Theorem \ref{theorem1}, we note that the only reason to require $\mathfrak{c}=\omega_2$ in the final model is to prove and apply Lemma \ref{lem candidate}. So in this section, we will find a new lemma to play the role of Lemma \ref{lem candidate} while the continuum can be enlarged.

  Here we will use the generalized bounding number for the eventual domination on $\omega_1$ (see \cite{CS}). For $f, g \in \omega_1^{\omega_1}$ functions  from  $\omega_1$ to $\omega_1$, say 
  $$f<^* g\text{ if }|\{\alpha<\omega_1: f(\alpha)\geq g(\alpha)\}|\leq \omega.$$
   $\mathfrak{b}_{\omega_1}$ is the minimal size of an unbounded subfamily of $(\omega_1^{\omega_1}, <^*)$.
  
  The following fact follows from \cite[Theorem 4.2]{Baumgartner}. See also \cite{CS} for a systematic study of $\mathfrak{b}_\kappa$ and related cardinal invariants.
  \begin{thm}[\cite{Baumgartner}; \cite{CS}]
  Assume {\rm GCH} and let $\kappa>\omega_2$ be regular. Then there is a $\sigma$-closed, $\omega_2$-cc poset $\mc{P}$ such that $\Vdash_\mc{P} {\mathrm{CH}} $ and $\mathfrak{b}_{\omega_1}=\kappa=\kappa^{<\kappa}$.
  \end{thm}
  
  So we will start from a model of {\rm CH} and $\mathfrak{b}_{\omega_1}=\kappa=\kappa^{<\kappa}$ for some $\kappa>\omega_2$. Fix a coloring $\pi: [\omega_1]^2\ra 2$ witnessing {\rm Pr}$_0(\omega_1, 2, \omega)$ (see \cite{Todorcevic89} for a proof of $\mathfrak{b}=\omega_1\Rightarrow \mathrm{Pr}_0(\omega_1, \omega, \omega)$).
  
  We will then iteratively force with ccc posets with finite support. We then need the fact that $\mathfrak{b}_{\omega_1}$ remains the same in the iteration process, since every function $f\in \omega_1^{\omega_1}$ in a ccc extension is dominated by some $g$ in the ground model. Just note that if $p\Vdash_\mc{P} \dot{f}(\alpha)<\omega_1$ where $\mc{P}$ is a ccc poset, then for some $\alpha'<\omega_1$, $p\Vdash_\mc{P} \dot{f}(\alpha)<\alpha'$.
  \begin{fact}
  Suppose $\mc{P}$ is a ccc poset and $\dot{f}$ is a $\mc{P}$-name of a function from $\omega_1$ to $\omega_1$. Then there is a function $g: \omega_1\ra \omega_1$ in the ground model such that $\Vdash_{\mc{P}} \dot{f}(\alpha)< g(\alpha)$ for every $\alpha<\omega_1$.
  \end{fact}
  
  The following lemma will be used to play the role of Lemma \ref{lem candidate}.
   \begin{lem}\label{lem candidate1}
   Assume $\varphi_0(\pi, \mc{C}, \mc{R}, \mc{T}, \mathbf{E})$ and $|\mct|< \mathfrak{b}_{\omega_1}$. For every $A\in [\omega_1]^{\omega_1}$, there exists $B\in [A]^{\omega_1}$ such that
   $$\text{ for every } B'\subset \Phi_\mct(B) \text{ in } \mct, |B\cap B'|\leq \omega.$$
  \end{lem}
  \begin{proof}
  First assume that for some $C\in \mct$, 
  \begin{enumerate}
  \item $|A\cap C|=\omega_1$ and for every countable subset $\mct'$ of $\{C'\in \mct: C'\subset C\}$, $|A\cap C\setminus \bigcup \mct'|=\omega_1$.
  \end{enumerate}
  Let 
  $$\mct'=\{C'\in succ_{\mct}(C):   |A\cap C'|=\omega_1\}$$
  where   $succ_{\mct}(C)$ is the collection of immediate successors of $C$ in $\mct$.
    
  If $|\mct'|<\omega_1$, then take $B=A\cap C\setminus \bigcup \mct'$. Clearly,  $\Phit(B)=C$ and $B$ is as desired.
  
  Now suppose $\mct'$ is uncountable. Fix $\{C_\alpha\in \mct': \alpha<\omega_1\}$.
  
  For every $D\in succ_\mct(C)\setminus \{C_\alpha: \alpha<\omega_1\}$, define $f_D: \omega_1\ra \omega_1$ by
  $$f_D(\alpha)=\sup (D\cap C_\alpha).$$
  Here we denote $\sup (\emptyset)=0$. By (T3), $f_D$ is a function from $\omega_1$ to $\omega_1$.
  
  Since $|\mct|< \mathfrak{b}_{\omega_1}$, there is $g\in \omega_1^{\omega_1}$ such that $f_D<^* g$ for all $D\in succ_\mct(C)\setminus \{C_\alpha: \alpha<\omega_1\}$. Now inductively choose $x_\alpha\in A\cap C_\alpha\setminus g(\alpha)$ such that 
  $$x_\alpha> \sup\{x_\xi: \xi<\alpha\} \text{ and } x_\alpha\notin \bigcup_{\xi<\alpha} C_\xi.$$
  Let $B=\{x_\alpha: \alpha<\omega_1\}$. Clearly, $\Phit(B)=C$ and $B$ is as desired.\medskip
  
  Recall that the proof of Lemma \ref{lem candidate} shows the existence of $C$ satisfying (1). This finishes the proof of the lemma. 
  \end{proof}
  
  Repeating the proof of Theorem \ref{theorem1} with Lemma \ref{lem candidate} replaced by Lemma \ref{lem candidate1}, we get the following.
       \begin{thm}\label{thm2}
  Suppose $\mathfrak{b}_{\omega_1}=\kappa=\kappa^{<\kappa}$. Then there is a finite support iteration of ccc posets $\langle \mc{P}_\alpha, \dot{\mc{Q}}_\beta: \alpha\leq \kappa, \beta<\kappa\rangle$  such that in $V[G]$ where $G$ is $\mc{P}_{\kappa}$-generic,
  \begin{enumerate}
  \item {\rm MA}(powerfully ccc) holds  and $2^\omega=\kappa$;
  \item {\rm MA} fails.
  \end{enumerate}
  \end{thm}
 CH may fail in the ground model in above theorem, but we can first force with Cohen forcing to get a coloring witnessing Pr$_0(\omega_1, 2, \omega)$. Then we get a model of $\mathfrak{b}_{\omega_1}=\kappa=\kappa^{<\kappa}$ + Pr$_0(\omega_1, 2, \omega)$.

  \section{Strengthening {\rm MA}(powerfully ccc)}
  In this section, we will prove that $\neg${\rm MA} is consistent with several strengthening of \map. The first strengthened forcing axiom is  {\rm MA}(squarely ccc). 
  
  Recall that a poset is \emph{squarely ccc} if its square is ccc.
  
  \subsection{{\rm MA}(squarely ccc)}
  
 The consistency of {\rm MA}(squarely ccc)+$\neg${\rm MA}  also indicates to what extent can a ccc poset be different from being a square. More precisely, we will construct a model in which there is a ccc poset without property ($K$) while the square of any ccc poset is either non-ccc or has precalibre $\omega_1$.
  
  Recall that in the proof of Theorem \ref{theorem1}, we need to destroy ccc of $\mc{Q}^m$ for some finite $m$ according to Lemma \ref{lem25} if a ccc poset $\mc{Q}$ does not preserve $\phipi$.  And the value of $m$ is determined by Lemma \ref{lem24}. If we can modify the proof and always choose $m$ to be 2, then we get {\rm MA}(squarely ccc) together with $\mc{H}_0, \mc{H}_1$ being ccc.
  
  We will use the following lemma to play the role of Lemma \ref{lem24}.
  \begin{lem}\label{lem30}
  Assume $\varphi(\pi, \mc{C}, \mc{R}, \mct, \bfe)$ and $\psi_0(\pi, \mc{C}, \mc{R}, \mct, \bfe)$.  Suppose  $\msa$ is a candidate for $\mc{C}$. Then there are $m<\omega$ and a sequence of $\mc{R}$-satisfiable for $\msa$ functions $\langle f_{i, j}: i, j<m\rangle$ such that for all 
  $$a_0<a_1<\cdot\cdot\cdot <a_{m-1}<b_0<\cdot\cdot\cdot<b_{m-1}$$
   in $\msa$, if  $f_{i, j}$ is order isomorphic to $\pi\up_{a_i\times b_j}$ whenever $i,j<m$, then for every $l<2$ and $b\in \mc{H}_l$, either $b\cap a_i=\emptyset$ or $b\cap b_i=\emptyset$ for some $i<m$.
  \end{lem}
  \begin{proof}
  Take $m=2N_\msa^{N_\msa}$. Let $\{s_k: k<N_\msa^{N_\msa}\}$ enumerate functions from $N_\msa$ to $N_\msa$.
  
  Then for $i<N_\msa$ and $j<N_\msa^{N_\msa}$, let $f_{i, j}$ be a $\mc{R}$-satisfiable for $\msa$ function such that $f_{i, j}(s_j(i), i)=0$ and $f_{i, N_\msa^{N_\msa}+j}$ be a $\mc{R}$-satisfiable for $\msa$ function such that $f_{i, N_\msa^{N_\msa}+j}(s_j(i), i)=1$. The existence of $f_{i, j}$'s follows from $\psi_0(\pi, \mc{C}, \mc{R}, \mct, \bfe)$. The rest $f_{i, j}$'s are arbitrary $\mc{R}$-satisfiable for $\msa$ functions.
  
  To show that $m$ and $\langle f_{i, j}: i, j<m\rangle$ are as desired, fix 
    $$a_0<a_1<\cdot\cdot\cdot <a_{m-1}<b_0<\cdot\cdot\cdot<b_{m-1}$$
     in $\msa$ such that  $f_{i, j}$ is order isomorphic to $\pi\up_{a_i\times b_j}$ whenever $i,j<m$. Fix $l<2$ and $b\in \mc{H}_l$.
  
  If $b\cap a_i=\emptyset$ for some $i<N_\msa$, then we are done.
  
  Suppose otherwise. Define $s: N_\msa\ra N_\msa$ by 
  $$a_i(s(i))=\min (b\cap a_i) \text{ for all }i<N_\msa.$$
  Assume $s=s_j$ for some $j<N_\msa^{N_\msa}$.
  
  If $l=1$, then $b\cap b_{j}=\emptyset$. To see this, note that $b_{j}(i)\notin b$ since 
  $$\pi(a_i(s_j(i)), b_{j}(i))=f_{i, j}(s_j(i), i)=0\text{ and }a_i(s_j(i))=a_i(s(i))\in b.$$
  
  Similarly, if $l=0$, then $b\cap b_{ N_\msa^{N_\msa}+j}=\emptyset$. This finishes the proof of the lemma.
  \end{proof}
  
  Now Lemma \ref{lem25} can be strengthened to requiring $m=2$. The proof is similar to the proof of Lemma \ref{lem25}. So we only sketch the modifications.
   \begin{lem}\label{lem31}
  Assume $\varphi(\pi, \mc{C}, \mc{R}, \mct, \bfe)$, $\psi_0(\pi, \mc{C}, \mc{R}, \mct, \bfe)$ and $|\mct|<\mathfrak{b}_{\omega_1}$. Suppose $\mc{Q}$ is a ccc poset which forces the failure of $\varphi(\pi, \mc{C}, \mc{R}, \mct, \bfe)$. Then there are a candidate for $\mc{C}$ family $\msa$ and a non-empty collection of $\mc{R}$-satisfiable functions $H$ such that in $V[G]$ where $G$ is $\pah$-generic, $\mc{Q}^2$ is not ccc and $\psi_0(\pi, \mc{C}\cup\{\bigcup G\}, \mc{R}\cup\{(\bigcup G, H)\}, \mct', \bfe')$ holds
where $\mct'$ and $\bfe'$ are induced from $\mc{C}\cup\{\bigcup G\}$.
  \end{lem}
  \begin{proof}[Sketch proof.]
  Find $q$, $\dot{\msa}$, $n$, $f$, $\{(q'_\alpha, a'_\alpha): \alpha<\omega_1\}$ and $\Gamma\in [\omega_1]^{\omega_1}$, $\msa'=\{a'_\alpha: \alpha\in \Gamma\}$ as in the proof of Lemma \ref{lem25}.
  
Applying Lemma \ref{lem30} to $\msa'$, we find $m$ and a sequence of $\mc{R}$-satisfiable for $\msa'$ functions $\langle f_{i, j}: i, j<m\rangle$.

Inductively find $F_\alpha\in [\Gamma]^m$ such that 
  \begin{enumerate}
  \item $F_\alpha(0)>\sup \bigcup_{\xi<\alpha} F_\xi$ and  $\{q'_\xi: \xi\in F_\alpha\}$ has a common lower bound $q''_\alpha$. Denote $a''_\alpha=\bigcup_{\xi\in F_\alpha} a'_\xi$.
  \end{enumerate}
  Since $\msa'$ is a $\mc{C}$-candidate, $\msb^\alpha=\{a''_\beta: \alpha\leq\beta<\omega_1\}$ is also a $\mc{C}$-candidate  for each $\alpha<\omega_1$. Moreover, $f^*: nm\times nm\ra 2$ is $\mc{R}$-satisfiable for $\msb^\alpha$ for all $\alpha$ where 
  $$f^*\up_{[ni, n(i+1))\times [nj, n(j+1))}\text{ is order isomorphic to $f_{i,j}$ for all }i, j<m.$$
  
  Now applying $\phipi$ to $\msb^\alpha$ and $f^*$, we get $\alpha''> \alpha'\geq \alpha$ such that 
  $$\pi\up_{a''_{\alpha'}\times a''_{\alpha''}}\text{ is order isomorphic to }f^*.$$
  Let
  $$a_\alpha=a''_{\alpha'}\cup a''_{\alpha''} ,\ a_{\alpha, i}=a_\alpha[[ni, n(i+1))]\text{ for $i<2m$ and}$$
  $$p_\alpha=q''_{\alpha'}, \ q_\alpha=q''_{\alpha''}.$$
  By our choice of $f^*$, for every $\alpha<\omega_1$,
  \begin{enumerate}\setcounter{enumi}{1}
\item for every $l<2$ and $b\in \mc{H}_l$, $b\cap a_{\alpha, i}=\emptyset$ for some $i<2m$.
\end{enumerate}

As in the proof of Lemma \ref{lem25}, find $\Sigma\in [\omega_1]^{\omega_1}$ such that
 \begin{enumerate}\setcounter{enumi}{2}
\item $\msb=\{a_\alpha: \alpha\in \Sigma\}$ is a candidate for $\mc{C}$.
\end{enumerate}
Let $H$ be the collection of all $h: 2nm\times 2nm\ra 2$ such that
\begin{enumerate}\setcounter{enumi}{3}
  \item $h$ is a $\mc{R} $-satisfiable for $\msb$ function;
  \item for some $l<2m$, $h\up_{[nl, n(l+1))\times [nl, n(l+1))}$ is order isomorphic to $f$.
\end{enumerate}
Then the proof of Lemma \ref{lem25} shows that $\psi_0(\pi, \mc{C}\cup\{\bigcup G\}, \mc{R}\cup\{(\bigcup G, H)\}, \mct', \bfe')$ holds in $V[G]$ where $G$ is $\ms{P}_{\msb, H}$-generic.  Also, $\{(p_\alpha, q_\alpha): a_\alpha\in \bigcup G\}$ is an uncountable antichain of $\mc{Q}^2$.
   \end{proof}
  
  Repeating the proof of Theorem \ref{thm2} with Lemma \ref{lem25} replaced by Lemma \ref{lem31}, we get the following.
       \begin{thm}\label{thm3}
  Suppose $\mathfrak{b}_{\omega_1}=\kappa=\kappa^{<\kappa}$. Then there is a finite support iteration of ccc posets $\langle \mc{P}_\alpha, \dot{\mc{Q}}_\beta: \alpha\leq \kappa, \beta<\kappa\rangle$  such that in $V^{\mc{P}_{\kappa}}$,
  \begin{enumerate}
  \item {\rm MA}(squarely ccc) holds  and $2^\omega=\kappa$;
  \item {\rm MA} fails.
  \end{enumerate}
  \end{thm}
  
  \subsection{Forcing axiom for more ccc posets}
  
  We would like the final model to satisfy forcing axiom for more ccc posets while preserve ccc of $\mc{H}_0$ and $\mc{H}_1$. Since {\rm MA}(squarely ccc) can be consistently true by Theorem \ref{thm3}, we will look into ccc posets themselves. More precisely, we will investigate for what ccc poset $\mc{Q}$ that does not preserve $\phipi$, we can destroy ccc of $\mc{Q}$ by some $\pah$ while preserve extended $\psi_0(\pi, \mc{C}', \mc{R}', \mct', \bfe')$.
  
  The proof of Lemma \ref{lem31} can be modified to destroy ccc of $\mc{Q}$ if there are uncountably many $q_\alpha$'s such that for some $m$ and $\{a_{\alpha, i}: i< m\}$,
  \begin{itemize}
  \item[(i)] $q_\alpha\Vdash a_{\alpha ,i}\in \dot{\msa}\cap [\omega_1\setminus \alpha]^n$ and $a_{\alpha, 0}<\cdot\cdot\cdot<a_{\alpha, m-1}$;
  \item[(ii)] for every $l<2$ and $b\in \mc{H}_l$, $b\cap a_{\alpha, i}=\emptyset$ for some $i<m$.
  \end{itemize}
  
  So we investigate the property to guarantee the existence of $a_{\alpha, i}$'s with above properties.
  \begin{lem}\label{lem33}
  Suppose $\msa\subset [\omega_1]^{<\omega}$ is uncountable non-overlapping, $\pi: [\omega_1]^2\ra 2$ is a coloring and $l<2$. Then the following statements are equivalent.
  \begin{itemize}
  \item[(i)] For some $\delta<\omega_1$, for every finite $\msa'\subset \msa\cap [\omega_1\setminus \delta]^{<\omega}$,  there exists  $b\in \mc{H}_l$ such that $b\cap a\neq \emptyset$   whenever $a\in \msa'$.
  \item[(ii)] There is an $l$-homogeneous subset meeting all but countably many elements of $\msa$.
  \end{itemize}
  \end{lem}
  \begin{proof}
  Only ``(i)$\Rightarrow$(ii)'' needs a proof. Fix $\delta<\omega_1$ guaranteed by (i).
  
  Fix a non-principal ultrafilter $\mc{U}$ on $\omega$ and a uniform ultrafilter $\mc{V}$ on $\omega_1$, i.e., $|F|=\omega_1$ whenever $F\in \mc{V}$.
  
  We first show the countable case.\medskip
  
  \textbf{Claim.} If $\msb\subset \msa\cap [\omega_1\setminus \delta]^{<\omega}$ is countable, then there is an $l$-homogeneous set meeting all elements of $\msb$.
  \begin{proof}[Proof of Claim.]
  Without loss of generality, we may assume that $\msb$ is infinite and $\msb=\{a_n: n<\omega\}$. For every $n<\omega$, fix $b_n\in \mc{H}_l$ such that $b_n\cap a_m\neq \emptyset$ whenever $m<n$.
  
  For every $m<\omega$, let $x_m$ be the least $\alpha\in a_m$ such that $\{n: \alpha\in b_n\}\in \mc{U}$. Since $a_m$ is finite and $\mc{U}$ is non-principal, $x_m$ exists.
  
  Then $\{x_m: m<\omega\}$ is $l$-homogeneous and meeting all elements of $\msb$. To see homogeneity, note that $x_m$ and $x_k$ belong to the same $b_n$ for almost all $n$ in $\mc{U}$.
  \end{proof}
  
  Now for every $\alpha<\omega_1$, fix an $l$-homogeneous set $B_\alpha$ meeting all elements of $\msa\cap [\alpha\setminus \delta]^{<\omega}$.
  
    For every $a\in \msa\cap [\omega_1\setminus \delta]^{<\omega}$, let $y_a$ be the least $\xi\in a$ such that $\{\alpha: \xi\in B_\alpha\}\in \mc{V}$.  Then $\{y_a: a\in \msa\cap [\omega_1\setminus \delta]^{<\omega}\}$ is an $l$-homogeneous set meet all elements of $\msa\cap [\omega_1\setminus \delta]^{<\omega}$.
  \end{proof}
  
  Now the proof of Lemma \ref{lem31} shows the following.
   \begin{lem}\label{lem34}
  Assume $\varphi(\pi, \mc{C}, \mc{R}, \mct, \bfe)$, $\psi_0(\pi, \mc{C}, \mc{R}, \mct, \bfe)$ and $|\mct|<\mathfrak{b}_{\omega_1}$. Suppose $\mc{Q}$ is a ccc poset which forces the failure of $\varphi(\pi, \mc{C}, \mc{R}, \mct, \bfe)$. If $\mc{Q}$ forces that there is no uncountable $l$-homogeneous subset for any $l<2$, then there are a candidate for $\mc{C}$ family $\msa$ and a non-empty collection of $\mc{R}$-satisfiable functions $H$ such that in $V[G]$ where $G$ is $\pah$-generic, $\mc{Q}$ is not ccc and $\psi_0(\pi, \mc{C}\cup\{\bigcup G\}, \mc{R}\cup\{(\bigcup G, H)\}, \mct', \bfe')$ holds
where $\mct'$ and $\bfe'$ are induced from $\mc{C}\cup\{\bigcup G\}$.
  \end{lem}
  \begin{proof}[Sketch proof.]
  Find $q$, $\dot{\msa}$, $n$, $f$ as in the proof of Lemma \ref{lem25}. By Lemma \ref{lem33}, for eveary $\alpha<\omega_1$, find $q_\alpha< q$, $m_\alpha<\omega$ and $\{a_{\alpha, i}: i<m_\alpha\}$ such that
  \begin{enumerate}
  \item $q_\alpha\Vdash a_{\alpha, 0}<\cdot\cdot\cdot<a_{\alpha, m_\alpha-1}$ are in $\dot{\msa}\cap [\omega_1\setminus \alpha]^n$;
  \item  for every $l<2$ and $b\in \mc{H}_l$, $b\cap a_{\alpha, i}=\emptyset$ for some $i<m_\alpha$.  
    \end{enumerate}
  Find $\Gamma\in [\omega_1]^{\omega_1}$ and $m$ such that for every $\alpha\in \Gamma$,
  \begin{enumerate}\setcounter{enumi}{2}
  \item $m_\alpha=m$, $a_\alpha=\bigcup_{i<m} a_{\alpha, i}$ and $a_\xi<a_\alpha$ whenever $\xi\in \Gamma\cap \alpha$.
  \end{enumerate}
  As in the proof of Lemma \ref{lem25}, find $\Gamma'\in [\Gamma]^{\omega_1}$ such that 
   \begin{enumerate}\setcounter{enumi}{3}
  \item $\msb=\{a_\alpha: \alpha\in \Gamma'\}$ is a candidate for $\mc{C}$.
  \end{enumerate}

Let $H$ be the collection of all $h: nm\times nm\ra 2$ such that
\begin{enumerate}\setcounter{enumi}{4}
  \item $h$ is $\mc{R} $-satisfiable for $\msb$ functions;
  \item for some $l<m$, $h\up_{[nl, n(l+1))\times [nl, n(l+1))}$ is order isomorphic to $f$.
\end{enumerate}
Then the proof of Lemma \ref{lem25} shows that $\psi_0(\pi, \mc{C}\cup\{\bigcup G\}, \mc{R}\cup\{(\bigcup G, H)\}, \mct', \bfe')$ holds in $V[G]$ where $G$ is $\ms{P}_{\msb, H}$-generic.  Also, $\{q_\alpha: a_\alpha\in \bigcup G\}$ is an uncountable antichain of $\mc{Q}$.
   \end{proof}
  
    Modifying the proof of Theorem \ref{thm3}, we get the following.
           \begin{thm}\label{thm4}
         Suppose $\mathfrak{b}_{\omega_1}=\kappa=\kappa^{<\kappa}$ and $\pi: [\omega_1]^2\ra 2$ satisfies {\rm Pr}$_0(\omega_1, 2, \omega)$. Then there is a finite support iteration of ccc posets $\langle \mc{P}_\alpha, \dot{\mc{Q}}_\beta: \alpha\leq \kappa, \beta<\kappa\rangle$  such that in $V^{\mc{P}_{\kappa}}$,
  \begin{enumerate}
  \item {\rm MA}(squarely ccc) holds and $2^\omega=\kappa$;
  \item $\mc{H}_0$ and $\mc{H}_1$ are ccc and in particular, {\rm MA} fails;
  \item if $\mc{P}$ is a ccc poset which forces that $\pi$ has no uncountable $l$-homogeneous subset for any $l<2$, then \ma($\{\mc{P}\}$) holds.
  \end{enumerate}
  \end{thm}

   \subsection{Iterating powerfully ccc posets}
  Although {\rm MA}(powerfully ccc) is highly unlikely to imply preservation of powerfully ccc in finite support iteration,  in this subsection, we introduce several properties that, together with {\rm MA}(powerfully ccc), imply preservation of powerfully ccc in finite support iteration.  The consistency of these properties follows from the following theorem.

         \begin{thm}\label{thm6}
         Suppose $\mathfrak{b}_{\omega_1}=\kappa=\kappa^{<\kappa}$ and $\pi: [\omega_1]^2\ra 2$ satisfies {\rm Pr}$_0(\omega_1, 2, \omega)$. Then there is a finite support iteration of ccc posets $\langle \mc{P}_\alpha, \dot{\mc{Q}}_\beta: \alpha\leq \kappa, \beta<\kappa\rangle$  such that in $V^{\mc{P}_{\kappa}}$,
  \begin{enumerate}
  \item {\rm MA}(squarely ccc) holds and $2^\omega=\kappa$;
  \item $\mc{H}_0$ and $\mc{H}_1$ are ccc and in particular, {\rm MA} fails;
  \item if $\mc{P}$ is a ccc poset which forces that $\pi$ has no uncountable $l$-homogeneous subset for any $l<2$, then \ma($\{\mc{P}\}$) holds;
  \item for every $A\in [\omega_1]^{\omega_1}$ and $l<2$, there are $m<\omega$ and uncountable non-overlapping $\msb\subset [A]^m$ such that for every $a<b$ in $\msb$, $\pi(a(i), b(i))=l$ for some $i<m$.
  \end{enumerate}
  \end{thm}
  \begin{proof}
  Let $\langle \mc{P}_\alpha, \dot{\mc{Q}}_\beta: \alpha\leq \kappa, \beta<\kappa\rangle$ be a standard iteration of ccc posets of size $<\kappa$ such that at step $\alpha$,
  \begin{itemize}
  \item[(i)] If $\alpha$ is a limit ordinal, then choose $\dot{\mc{Q}}_\alpha$ to guarantee that (1)-(3) hold in the final model according to the construction in Theorem \ref{thm4};
  \item[(ii)]   If $\alpha$ is a successor ordinal, then we deal with $\dot{A}$ in $[\omega_1]^{\omega_1}$ and $l<2$ for $\dot{A}, l$ in some standard enumeration.
  \end{itemize} 
Then (1)-(3) above hold   and we describe the procedure at successor steps.
  
Fix a successor ordinal $\alpha<\kappa$.  Choose a $\mc{P}_\alpha$-generic filter $G$. We work in $V[G]$. 

Fix $A\in [\omega_1]^{\omega_1}$ and $l<2$ as in (ii). Since $|\mct^\alpha|<\kappa=\mathfrak{b}_{\omega_1}$,  by Lemma \ref{lem candidate1}, find $B\in [A]^{\omega_1}$ such that
\begin{itemize}
\item[(a)] for every $B'\subset \Phi_{\mct^\alpha}(B)$ in $\mct^\alpha$, $|B'\cap B|\leq \omega$.
\end{itemize}
Fix uncountable non-overlapping $\msa\subset [\omega_1]^{N_\msa}$ and $i^*<N_\msa$ such that
\begin{itemize}
\item[(b)] $\{\msa_j: j<N_\msa\}$ is the $\bfe^\alpha$-closure of $B$ and $B=\msa_{i^*}$.
\end{itemize}
By Lemma \ref{lem7} and Lemma \ref{lem10}, going to an uncountable subset if necessary, we may assume that
\begin{itemize}
\item[(c)] $\msa$ is a $\mc{C}^\alpha$-candidate.
\end{itemize}
Consider poset  $\mc{Q}=\{F\in [\msa]^{<\omega}: \{a({i^*}): a\in F\}\in \mc{H}_{1-l}\}$ ordered by reverse inclusion.

Note that $\mc{Q}$ is ccc since $\mc{H}_{1-l}$ is ccc. Also, $\mc{Q}$ forces the failure of $\varphi(\pi, \mc{C}^\alpha, \mc{R}^\alpha, \mct^\alpha, \bfe^\alpha)$. In fact, by $\psi_0(\pi, \mc{C}^\alpha, \mc{R}^\alpha, \mct^\alpha, \bfe^\alpha)$, let $f: n\times n\ra 2$ be $\mc{R}^\alpha$-satisfiable for $\msa$ where $n=N_\msa$ such that
\begin{itemize}
\item[(d)] $f(i^*, i^*)=l$. Then $f$ can not be realized in $\bigcup G'$ where $G'$ is $\mc{Q}$-generic and uncountable.
\end{itemize}
By (d), for some $q\in \mc{Q}$, the $\mc{Q}$-name $\dot{\msa'}$ of $\bigcup G'$,  $q\Vdash_\mc{Q} \dot{\msa'}, f$ witness the failure of condition (Res) in $\varphi(\pi, \mc{C}^\alpha, \mc{R}^\alpha, \mct^\alpha, \bfe^\alpha)$ and $\dot{G'}$ is uncountable.

Choose, for every $\beta<\omega_1$, $a'_\beta\in \msa\cap [\omega_1\setminus \beta]^n$ such that $q'_\beta=q\cup \{a'_\beta\}\in \mc{Q}$. Then, for every $\beta<\omega_1$,
\begin{itemize}
\item[(e)] $q'_\beta\Vdash_\mc{Q} a'_\beta\in \dot{\msa'}\cap [\omega_1\setminus \beta]^n$.
\end{itemize}
Now repeating the proof of Lemma \ref{lem25}, we get $m$, $\msa'=\{a_\beta: \beta<\omega_1\}$, $\msa''\in [\msa']^{\omega_1}$ and $H$ such that    
\begin{itemize}
\item[(f)] $a_\beta=\bigcup_{i<m} a_{\beta, i}$ where each $a_{\beta, i}$ equals some $a'_{\eta}$ and $a_\beta<a_\gamma$ for $\beta<\gamma$;
\item[(g)] $\psi_0(\pi, \mc{C}^\alpha\cup \{\bigcup G''\}, \mc{R}^\alpha\cup \{(\bigcup G'', H)\}, \mct', \bfe')$ holds in $V[G][G'']$ where $G''$ is $\mc{P}_{\msa'', H}$-generic and $\mct', \bfe'$ are induced from $\mc{C}^\alpha\cup \{\bigcup G''\}$;
\item[(h)] for every $a_\beta<a_\gamma$ in $\bigcup G''$, there is $j<m$ such that $\pi(a_{\beta, j}(i^*), a_{\gamma, j}(i^*))=l$.
\end{itemize}
To see (h), note that $f(i^*, i^*)=l$ and  $\pi\up_{a_{\beta, j}\times a_{\gamma, j}}$ is order isomorphic to $f$ for some $j<m$.

Now $m$ and $\msb=\{\{a_{\beta, j}(i^*): j<m\}: a_\beta\in \bigcup G''\}$ satisfy the following,
\begin{itemize}
\item[(i)] for   $a<b$ in $\msb$, $\pi(a(j), b(j))=l$ for some $j<m$.
\end{itemize}
Take $\mc{P}_{\msa'', H}$ as $\mc{Q}_\alpha$. It is straightforward to check that (4) holds.
  \end{proof}
  
    \begin{thm}\label{thm7}
         Suppose {\rm MA}(powerfully ccc) + $\neg\mathrm{CH}$ and for coloring $\pi: [\omega_1]^2\ra 2$, $\mc{H}_0$ and $\mc{H}_1$ are ccc. Suppose moreover that,
  \begin{itemize}
  \item[(i)] if $\mc{P}$ is a ccc poset which forces that $\pi$ has no uncountable $l$-homogeneous subset for any $l<2$, then \ma($\{\mc{P}\}$) holds;
  \item[(ii)] for every $A\in [\omega_1]^{\omega_1}$ and $l<2$, there are $m<\omega$ and uncountable non-overlapping $\msb\subset [A]^m$ such that for every $a<b$ in $\msb$, $\pi(a(i), b(i))=l$ for some $i<m$.
  \end{itemize}
  Then powerfully ccc is closed under finite support iterated forcing.
  \end{thm}
  \begin{proof}
  Let $\langle \mc{P}_\alpha, \dot{\mc{Q}}_\beta: \alpha\leq \nu, \beta<\nu\rangle$ be a finite support iteration of powerfully ccc posets. We will show that $\mc{P}_\nu$ is powerfully ccc by induction on $\nu$. First note that, since {\rm MA}(powerfully ccc) holds, a poset is powerfully ccc   iff it has precaliber $\omega_1$.
  
  $\nu=1$ is trivial.
  
  Then assume that $\nu=\eta+1$ is a successor ordinal. If {\rm MA}$(\{\mc{P}_\nu(\leq r)\})$ holds for every $r\in \mc{P}_\nu$ where $\mc{P}_\nu(\leq r)=\{r'\in \mc{P}_\nu: r'\leq r\}$, then   $\mc{P}_\nu$ has precaliber $\omega_1$.
  
  So suppose otherwise. View $\mc{P}_\nu$ as $\mc{P}_\eta*\dot{Q}_\eta$. By (i), for some $l<2$, $(p, \dot{q})\in\mc{P}_\nu$ and $\mc{P}_\nu$-name $\dot{X}$,
  \begin{enumerate}
  \item $(p,\dot{q})\Vdash_{\mc{P}_\nu} \pi\text{ has an uncountable $l$-homogeneous subset } \dot{X}.$
  \end{enumerate}
  For each $\alpha<\omega_1$, find $(p_\alpha, \dot{q}_\alpha)\leq (p,\dot{q})$ and $x_\alpha\in \omega_1\setminus \alpha$ such that
  \begin{enumerate}\setcounter{enumi}{1}
  \item $(p_\alpha, \dot{q}_\alpha)\Vdash_{\mc{P}_\nu} x_\alpha\in \dot{X}$.
  \end{enumerate}
    By induction hypothesis, $\mc{P}_\eta$ is powerfully ccc and hence has precaliber $\omega_1$.  Then find $\Gamma\in [\omega_1]^{\omega_1}$ such that
      \begin{enumerate}\setcounter{enumi}{2}
  \item $\{p_\alpha: \alpha\in\Gamma\}$ is centered. 
  \end{enumerate}
  Denote
  $$A=\{x_\alpha: \alpha\in \Gamma\}.$$
  By (ii), there are $m<\omega$ and uncountable non-overlapping $\msb\subset [A]^m$ such that 
   \begin{enumerate}\setcounter{enumi}{3}
  \item for every $a<b$ in $\msb$, $\pi(a(i), b(i))=1-l$ for some $i<m$.
  \end{enumerate}
  By (3), for every $a\in \msb$, find $p'_a\in \mc{P}_\eta$ such that
   \begin{enumerate}\setcounter{enumi}{4}
  \item $p'_a$ is a common lower bound of $\{p_\xi: \xi\in a\}$.
  \end{enumerate}
  Now choose a $\mc{P}_\eta$-generic filter $G$ such that 
  \begin{enumerate}\setcounter{enumi}{5}
  \item $\msa=\{a\in\msb: p'_a\in G\}$ is uncountable.
  \end{enumerate}
  
  For $a\in \msa$ and $i<m$, let 
  $$q_{a, i}=\dot{q}_{a(i)}^G.$$
  We claim that $\{\langle q_{a, i}: i<m\rangle: a\in \msa\}$ is an antichain of $\mc{Q}_\eta^m$ where $\mc{Q}_\eta=\dot{Q}_\eta^G$. To see this, fix $a<b$ in $\msa$. By (4),  $\pi(a(i), b(i))=1-l$ for some $i<m$. Then by (1) and (2), $(p_{a(i)},\dot{q}_{a(i)})$ is incompatible with $(p_{b(i)},\dot{q}_{b(i)})$. By (5) and (6), $p_{a(i)}$ and $p_{b(i)}$ are in $G$. So $q_{a, i}=\dot{q}_{a(i)}^G$ is incompatible with $q_{b,i}$.
  
  This shows that $\mc{Q}_\eta^m$ has an uncountable antichain in $V[G]$, contradicting the fact that $\Vdash_{\mc{P}_\eta} \dot{Q}_\eta$ is powerfully ccc. This finishes the successor case.
  
  Finally assume that $\nu$ is a limit ordinal. Note by induction hypothesis, $\mc{P}_\alpha$ is powerfully ccc and hence has precaliber $\omega_1$ for every $\alpha<\nu$. Now a standard $\Delta$-system argument shows that $\mc{P}_\nu$ has precaliber $\omega_1$.
  \end{proof}
Other than the strategy in Section 5, iteration of powerfully ccc posets according to above theorem can have arbitrary length not restricted by the ground model.

   \subsection{More variants}
 We introduce two more variants and leave the proof to interested reader.
  First,   preserve ccc of $\mc{H}_0$ only and require in addition that $\pi$ has no uncountable 0-homogeneous subset. 
    
   The following   property   plays a role similar to $\psipi$ in Section 4.
  \begin{defn}
  Assume $\varphi(\pi, \mc{C}, \mc{R}, \mct, \bfe)$. Then $\psi_1(\pi, \mc{C}, \mc{R}, \mct, \bfe)$ is the assertion that  for every $\msa$ that is a candidate for $\mc{C}$, 
  \begin{itemize}
  \item[(i)]  for every $I, J\subseteq N_\msa$ such that $\{a[I] ,a[J]: a\in \msa\}\subset \mc{H}_0$, there is a $\mc{R}$-satisfiable for $\msa$ function $f$ such that $f[I\times J]=\{0\}$;
  \item[(ii)]   for every $i, j<N_\msa$, there is a $\mc{R}$-satisfiable for $\msa$ function $g$ such that $g(i,j)=1$.
  \end{itemize}
  \end{defn}
  
Then the following statements are consistent for $\pi: [\omega_1]^2\ra 2$.
  \begin{itemize}
  \item {\rm MA}(squarely ccc) holds and $2^\omega$ is arbitrarily large.
  \item $\mc{H}_0$ is ccc and $\pi$ has no uncountable 0-homogeneous subset. In particular, {\rm MA} fails.
  \item If $\mc{P}$ is a ccc poset which forces that $\pi$ has no uncountable 0-homogeneous subset, then \ma($\{\mc{P}\}$) holds.
  \end{itemize}

  Above statement can be strengthened to satisfy the following: for a ccc poset $\mc{P}$, $\mc{P}$ forces that $\mc{H}_0$ is not ccc iff $\mc{P}$ forces that $\pi$ has an uncountable 0-homogeneous subset.
  
  Second, for $n\geq 2$,   force $\mc{H}_0$ to have property $Pr_n$ (see \cite[Definition 4.1]{Bagaria} or \cite{BS}). So $\mathrm{MA}(Pr_n)$ fails in the final model.  We find a witness of $Pr_{n}(\mc{H}_0)$ as follows.
  
  For a $\mc{C}$-candidate $\msa$ and $I$ such that  $\msa^*=\{a[I]: a\in \msa\}\subset\mc{H}_0$,  iteratively force with ``appropriate'' $\mc{P}_{\msa^i, H_i}$ for $i<\frac{n(n-1)}{2}$  such that
\begin{itemize}
\item $G_i$ is $\mc{P}_{\msa^i, H_i}$ generic over $V^{\mc{P}_{\msa^0, H_0}*\cdot\cdot\cdot*\mc{P}_{\dot{\msa}^{i-1}, H_{i-1}}}$ and  every element $a$ of $\msa^i$ is a union of $k_i$ elements of $\bigcup G_{i-1}$ (view $\bigcup G_{-1}$ as $\msa$): $a=\bigcup_{j<k_i} a_{j}$;
\item for every $a<b$ in $\bigcup G_i$, there exists $j<k_i$ such that $a_j[J]\cup b_j[J]\in \mc{H}_0$ whenever $a_j[J]\in \msa^*$.
\end{itemize}
Then $\{[a]^{|I|}\cap \msa^*: a\in \bigcup G_{\frac{n(n-1)}{2}-1}\}$ is a witness of $Pr_{n}(\mc{H}_0)$ for $\msa^*$.

  \section{Closing remarks}
  {\rm MA}$_{\omega_1}$ can be viewed as the combination of the following two statements for a property $\Psi$ in Figure 1 below countable and above ccc,
  \begin{itemize}
  \item[(i)] every ccc poset of size $\leq \omega_1$ has property $\Psi$;
  \item[(ii)] {\rm MA}$_{\omega_1}(\Psi)$.
  \end{itemize}
  Now for no $\Psi$ does (ii) imply {\rm MA}$_{\omega_1}$. It is interesting to know for what $\Psi$ does (i) imply {\rm MA}$_{\omega_1}$. It is easy to see that (i) for powerfully ccc (or squarely ccc) is equivalent to (i) for productively ccc. And it is proved in \cite{TV} that 
  \begin{itemize}
  \item (i) for precaliber $\omega_1$ implies {\rm MA}$_{\omega_1}$;
  \item   $\mathscr{K}_{n+1}$  that every ccc poset has property $K_{n+1}$ implies (i) for $\sigma$-$n$-linked.
  \end{itemize}
 No other implication is known except for the trivial ones.   See \cite{TV},  \cite{Todorcevic91}, \cite{Moore2000} , \cite{LT2001}, \cite{MY} for more partial results.  The following question is still open.
     \begin{question}[\cite{LT2001}]
 Does $\mathscr{K}_2$ imply $\mathrm{MA}_{\omega_1}$? Does productivity of ccc  imply $\mathscr{K}_2$?
  \end{question}

  The modifications in Section 6 indicate that the iteration introduced in Section 3 may have potential to distinguish more closely related properties. It is possible that this iteration   will be useful in distinguishing (i) above for different $\Psi$'s.

 \section*{Acknowledgement}
  I would like to thank Stevo Todorcevic and Liuzhen Wu for introducing Question \ref{q2} and indicating its connection with Question \ref{q1} to me.

\bibliographystyle{plain}

\end{document}